\documentclass[11pt,reqno]{amsart}

\usepackage{amssymb,amsmath}
\usepackage[pagebackref,hypertexnames=false, colorlinks, citecolor=red, linkcolor=red]{hyperref} 
\usepackage[backrefs]{amsrefs}

\newcommand{\diam}{\mbox{\rm diam}}\newcommand{\be}{\begin{eqnarray}}
\newcommand{\ee}{\end{eqnarray}}

\newcommand{\R}{{\mathbb R}}

\newcommand{\B}{{\mathcal B}}

\newcommand{\al}{\alpha}
\newcommand{\Dk}{{\mathcal D}}

\newcommand{\La}{{\Lambda}}

\newcommand{\om}{\omega}

\newcommand{\cz}{Calder\'on--Zygmund}
\newcommand{\vf}{\varphi}
\newcommand{\pd}{\partial}
\newcommand{\dist}{\operatorname{dist}}\newcommand{\ci}[1]{_{{}_{\scriptstyle{#1}}}}\newcommand{\supp}{\operatorname{supp}}\newtheorem{theorem}{Theorem}\newtheorem{lemma}[theorem]{Lemma}\newtheorem{prop}[theorem]{Proposition}\theoremstyle{definition}\newtheorem{defi}[theorem]{Definition}\theoremstyle{remark}\newtheorem{remark}[theorem]{Remark}\numberwithin{equation}{section}\input epsf.sty
\begin{document}\thispagestyle{empty}

\title[Bergman-type Singular Integral Operators]{{Bergman-type Singular Integral Operators and the characterization of Carleson measures for Besov--Sobolev spaces on the complex ball}}

\author[A. Volberg]{Alexander Volberg$^{\dagger}$}
\address{Alexander Volberg, Department of  Mathematics\\ Michigan State University\\ East Lansing, MI USA 48824}
\email{volberg@math.msu.edu}
\address{Alexander Volberg, Department of Mathematics\\ University of Edinburgh\\ James Clerk Maxwell Building\\ The King's Buildings\\ Mayfield Road\\ Edinburgh Scotland EH9 3JZ}
\email{a.volberg@ed.ac.uk}
\thanks{$\dagger$ Research supported in part by a National Science Foundation DMS grant.}
 
\author[B. D. Wick]{Brett D. Wick}
\address{Brett D. Wick, School of Mathematics\\ Georgia Institute of Technology\\ 686 Cherry Street\\ Atlanta, GA USA 30332-0160}
\email{wick@math.gatech.edu}

\begin{abstract} 
The purposes of this paper are two fold.  First, we extend the method of non-homogeneous harmonic analysis of Nazarov, Treil and Volberg to handle ``Bergman--type'' singular integral operators.  The canonical example of such an operator is the Beurling transform on the unit disc.  Second, we use the methods developed in this paper to settle the important open question about characterizing the Carleson measures for the Besov--Sobolev space of analytic functions $B^\sigma_2$ on the complex ball of $\mathbb{C}^d$.  In particular, we demonstrate that for any $\sigma> 0$, the Carleson measures for the space  are characterized by a ``T1 Condition''.  The method of proof of these results is an extension and another application of the work originated by Nazarov, Treil and the first author.
\end{abstract}

\maketitle

\section{Introduction and Statements of results}
\label{1}

We are interested in \cz\, operators that do not satisfy the standard estimates.  In particular, these kernels will live in $\mathbb{R}^d$ but will only satisfy estimates as if they live in $\mathbb{R}^m$ for $m\leq d$.    More precisely, we are interested in \cz\, operators whose kernels satisfy the following estimates
$$
|k(x,y)|\le\frac{C_{CZ}}{|x-y|^m}\,,
$$
and
$$
|k(y,x)-k(y,x')|+|k(x,y)-k(x',y)|\le C_{CZ}\frac{|x-x'|^\tau}{|x-y|^{m+\tau}}
$$
provided that $|x-x'|\le\tfrac12|x-y|$, with some (fixed) $0<\tau\le 1$ and $0<C_{CZ}< \infty$.  Once the kernel has been defined, then we say that a $L^2(\R^d;\mu)$ bounded operator is a Calder\'on--Zygmund operator with kernel $k$ if,
$$
T_{\mu} f(x)=\int_{\mathbb{R}^d}k(x,y)f(y)d\mu(y)\quad\forall x\notin\supp f\,.
$$

If $k$ is a nice function, then the integral above can be defined for all $x$ and then it gives one of the Calder\'on--Zygmund operators with kernel $k$.  It is for this reason, that when given a not so nice Calder\'on--Zygmund kernel people consider the sequence of ``cut-off" kernels and the {\it uniform boundedness} of this sequence.  In the applications that we have in mind for this paper, all our \cz\, operators can be considered a priori bounded (for example in future arguments, we can always think that $\mu$ is compactly supported inside the (complex) unit ball), and we will be interested in {\it the effective bound}, in terms of the parameters $C_{CZ}, \tau$ and in terms of a certain $T1$ condition we explain below.  Frequently, this is the better view point to adopt instead of the ``cut-off'' approach.

For kernels that satisfy these types of estimates and for virtually arbitrary underlying measures, a deep theory has been developed in \cites{NTV1, NTV2, NTV3, NTV4, V}.  The essential core of this theory showed that in this situation, one can develop the majority of Calder\'on--Zygmund theory and study the boundedness of the associated singular integral operators via a ``T1 Condition''.  This theory has found applications in Tolsa's solution of the Painlev\'e Problem, \cites{Tolsa1, Tolsa2, Tolsa}.  The theory of non-homogeneous harmonic analysis has also come to find applications in geometric measure theory in the works of Tolsa \cite{TolsaPV}, Tolsa, Ruiz de Villa \cite{TolsaVilla}, Mayboroda, Volberg \cites{MV1, MV2}, Eiderman, Nazarov, and Volberg \cite{ENV}.  The main results of this paper provide yet another application of the methods of non-homogeneous harmonic analysis.

Having in mind the application to Carleson measures in the complex unit ball, we wish to extend the theory of Nazarov, Treil and the first author \cites{NTV1, NTV2, NTV3, NTV4, V} to the case of singular integral operators that arise naturally as ``Bergman--type'' operators.  These will be operators that will satisfy the Calder\'on--Zygmund estimates from above, but we (again having in mind the above mentioned application, see further) additionally allow them to have the following  property
$$
|k(x,y)|\le \frac1{\max (d(x)^m, d(y)^m)}\,,
$$
where $d(x):=\dist (x, \R^d\setminus H)$ and $H$ being an open set in $\R^d$.  The examples that the reader should keep in mind are the \cz\, kernels built from the following function 
$$
k(x,y)=\frac{1}{\left(1-x\cdot y\right)^m}
$$
where $H=\mathbb{B}_d$, the unit ball in $\mathbb{R}^d$.  These are the standard ``Bergman--type'' kernels that arise naturally when looking at complex and harmonic analysis questions on the unit disc $\mathbb{D}$ and more generally in several complex variables.  When we have a kernel $k$ that satisfies the \cz\, estimates and  the additional property of measuring ``distance to some open set'' as above,  we will let
$$
T_{\mu,m}(f)(x)=\int_{\mathbb{R}^d}k(x,y)f(y)d\mu(y)\,.
$$

In applications, we will be viewing the Calder\'on--Zygmund kernels that arise as living on certain fixed sets.  Accordingly, we will say that $k$ is a Calder\'on--Zygmund kernel on a closed $X\subset \R^d$ if $k(x,y)$ is defined only on $X\times X$ and the previous properties of $k$ are satisfied whenever $x, x', y\in X$.

Our first main result is the following theorem, providing a link of the work of Nazarov, Treil and Volberg to the context of these Bergman--type operators.

\begin{theorem}[Main Result 1]
\label{md}
Let $k(x,y)$ be a \cz\, kernel of order $m\leq d$ on  $X\subset \R^d$, with \cz\,  constants $C_{CZ}$ and $\tau$. Let $\mu$ be a probability  measure with compact support in 
$X$ and suppose that all balls such that $\mu(B(x,r)) >r^m$ lie in an open set $H$. Let also
$$
|k(x,y)|\le \frac1{\max (d(x)^m, d(y)^m)}\,,
$$
where $d(x):=\dist (x, \R^d\setminus H)$.  Finally, suppose that for all cubes $Q$ a ``$T1$ Condition'' holds for the operator $T_{\mu,m}$ with kernel $k$ and for the operator $T^*_{\mu,m}$ with kernel $k(y,x)$:
\begin{equation}
\label{T1}
\|T_{\mu,m}\chi_Q\|_{L^2(\R^d;\mu)}^2 \le A\,\mu(Q)\,,\, \|T^*_{\mu,m}\chi_Q\|_{L^2(\R^d;\mu)}^2 \le A\,\mu(Q).
\end{equation}
Then $\|T_{\mu,m}\|_{L^2(\R^d;\mu)\rightarrow L^2(\R^d;\mu)} \le C(A,m,d,\tau)$.
\end{theorem}

The balls for which we have $\mu(B(x,r))>r^m$ are called ``non-Ahlfors balls".  Non-Ahlfors balls are enemies, their presence make the estimate of \cz\, operators basically impossible (often the boundedness of a \cz\, operator implies that non-Ahlfors balls do not exist). The key hypothesis is that we can capture all the non-Ahlfors balls in some open set $H$.  To mitigate against this difficulty, we will have to suppose that our \cz\, kernels have an additional estimate in terms of the behavior of the distance to the complement of $H$ (namely that they are Bergman--type kernels).  

The method of proof of this theorem will be to use the tools of non-homogeneous harmonic analysis as developed by F. Nazarov, S. Treil, and the first author in the series of papers \cites{NTV1, NTV2, NTV3, NTV4} and further explained in the book by the first author \cite{V}.  The key innovation is the use of the Bergman--type kernels and the ability to control them in some appropriate fashion.

The proof of Theorem \ref{md} arose in an attempt to characterize the Carleson measures for the Besov--Sobolev spaces of analytic functions on the unit ball $\mathbb{B}_{2d}$ in $\mathbb{C}^d$.  Roughly speaking, the Besov--Sobolev space of analytic functions $B_2^\sigma(\mathbb{B}_{2d})$ is the collection of analytic functions on the unit ball such the derivative of order $\frac{d}{2}-\sigma$ belongs to the classical Hardy space $H^2(\mathbb{B}_{2d})$.  An important question in the study of these spaces is a characterization of the measures $\mu$ for which
$$
\int_{\mathbb{B}_{2d}}\vert f(z)\vert^2d\mu(z)\leq C(\mu)\Vert f\Vert_{B_2^\sigma(\mathbb{B}_{2d})}^2.
$$
These measures are typically called Carleson measures and the characterization that one seeks is in terms of ``natural'' geometric test.  These ideas are explained in more detail in Section \ref{Carleson}.

For the range $\frac{d}{2}\leq\sigma$, the characterization of the Carleson measures is a well-known simple geometric condition.  Roughly speaking, the $\mu$-measure of a ball should be comparable to the Lebesgue measure of the same ball to some appropriate power (depending on $\sigma$).  See for example \cite{Zhu} for the Hardy space and the Bergman space.  For the range $0\leq\sigma\leq\frac{1}{2}$ these Carleson measures in the complex ball were initially characterized by Arcozzi, Rochberg and Sawyer.  In \cite{ARS} Arcozzi, Rochberg and Sawyer developed the theory of ``trees'' on the unit ball and then demonstrated that the inequality they wished to prove was related to a certain two-weight inequality on these trees.  Once they have the characterization of the trees, they can then deduce the corresponding characterization for the space of analytic functions.  Again, in the range $0<\sigma\leq\frac{1}{2}$ a proof more in the spirit of what appears in this paper was obtained by E. Tchoundja, \cites{T1, T2}.  The characterization of Carleson measures via capacitary conditions can be found in work by Ahern and Cohn, \cite{AC}.

An important open question in the theory of Besov--Sobolev spaces was a characterization of the Carleson measures in the difficult range $\frac{1}{2}<\sigma<\frac{d}{2}$, see for example \cites{ARS2}.  It is \textit{important} to emphasize that key to both the approaches of Arcozzi, Rochberg, and Sawyer and that of Tchoundja was a certain ``positivity'' of a kernel.  The tools of non-homogeneous harmonic analysis developed by Nazarov, Treil and Volberg were specifically developed to overcome the difficulty when a natural singular integral operator lacks positivity.  It is not at all clear how (or if it is even possible) to adapt the methods of \cites{ARS, T1, T2} to address the more general case of $\frac{1}{2}<\sigma<\frac{d}{2}$. This leads to the second main result of this paper, and a new application of the theory of non-homogeneous harmonic analysis.  In $\mathbb{C}^d$, let $\mathbb{B}_{2d}$ denote the open unit ball and consider the kernels given by 
$$
K_m(z,w):= \textnormal{Re}\,\frac{1}{(1-\bar{z}\cdot w)^{m}},\quad\forall |z|\le 1, |w|\le 1\,. 
$$

\begin{theorem}[Characterization of Carleson Measures for Besov--Sobolev Spaces $B_2^\sigma(\mathbb{B}_{2d})$]
\label{Carl2}
Suppose that $0<\sigma$.  Let $\mu$ be a positive Borel measure in $\mathbb{B}_{2d}$.  Then the following conditions are equivalent:
\begin{itemize}
\item[(a)] $\mu$ is a $B_2^\sigma(\mathbb{B}_{2d})$-Carleson measure;
\item[(b)] $T_{\mu,2\sigma}:L^2(\mathbb{B}_{2d};\mu)\to L^2(\mathbb{B}_{2d};\mu)$ is bounded;
\item[(c)] There is a constant $C$ such that

\begin{itemize}
\item[(i)] $\|T_{\mu,2\sigma}\chi_Q\|_{L^2(\mathbb{B}_{2d};\mu)}^2 \le C\,\mu(Q)$ for all $\Delta$-cubes $Q$;
\item[(ii)] $\mu(B_{\Delta}(x,r))\leq C\,r^{2\sigma}$ for all balls $B_{\Delta}(x,r)$ that intersect $\mathbb{C}^{d}\setminus\mathbb{B}_{2d}$.
\end{itemize}
\end{itemize}
\end{theorem}
Above, the sets $B_{\Delta}$ are balls measured with respect to a naturally occurring metric in the problem.  The operator $T_{\mu,2\sigma}$ is a Bergman-type Calder\'on--Zygmund operator with respect to this metric $\Delta$ for which we can apply (an extended version) of Theorem \ref{md}.  And, the set $Q$ is a ``cube'' defined with respect to the metric $\Delta$.  See Section \ref{s2} for the exact definitions of these objects.

We remark, that when $0\leq\sigma\leq\frac{1}{2}$ this result was previously obtained in \cites{ARS, T1, T2} and when $\frac{d}{2}\leq\sigma$ the characterization is now classical, see for example the references in \cite{Zhu}.  The key contribution of this theorem is the characterization of Carleson measures in the range $\frac{1}{2}<\sigma<\frac{d}{2}$, the so called ``difficult range''.

In this paper, constants will be denoted by $C$ throughout, and can change from line to line.  Frequently, the parameters that the constant depend upon will be indicated.  Also, we use the standard notation that $X\lesssim Y$ to mean that there is an absolute constant $C$ such that $X\leq CY$, and $X\approx Y$ to indicate that both $X\lesssim Y$ and $Y\lesssim X$.  

The authors wish to thank a skilled and thorough referee who pointed to areas where the exposition and clarity of the article could be improved.

\section{Applications of the Main Theorem:  Proof of Theorem \ref{Carl2}}
\label{s2}
We now show how one can use the main result of this paper, Theorem \ref{md} (and its metric version), to deduce several interesting corollaries.

\subsection{Necessary and Sufficient Conditions for ``Bergman--type'' Operators}

We now point out the situation that motivated Theorem \ref{md}.  In $\mathbb{C}^d$, let $\mathbb{B}_{2d}$ denote the open unit ball and consider the kernels given by 
$$
K_m(z,w):= \textnormal{Re}\,\frac{1}{(1-\bar{z}\cdot w)^{m}},\quad\forall |z|\le 1, |w|\le 1\,. 
$$

Let $\mu$ be a probability measure with compact support contained in the spherical layer $1/2 \le |z|<1$ and in particular the support is strictly inside the ball.  We will see that this kernel satisfies the hypotheses of Theorem \ref{md} when $H=\mathbb{B}_{2d}$, but with respect to a certain (non-euclidean) quasi-metric $\Delta$.  It will be obvious that we have the estimate
$$
|K_m(z,w)|\leq\frac{1}{\max (d(z)^m,d(w)^m)}\,.
$$
In application of the Main Theorem we need the quantity $d(z)$ which was defined with respect to the metric of interest.  We will see that for $z\in \mathbb{B}_{2d}$ this non-euclidean $d(z)$ will still be $1-|z|$: $d(z):=\textnormal{dist}_{\Delta}(z,\mathbb{C}^{d}\setminus\mathbb{B}_{2d})=1-|z|$.  Since if $z,w\in\mathbb{B}_{2d}$ we have $|1-z\cdot w|^{m}\geq (1-|z|)^{m}$ and a similar statement holding for $w$.

We introduce the above mentioned (quasi)-metric on the spherical layer around $\partial\mathbb{B}_{2d}$:
$$
\Delta(z,w):=
\left||z|-|w|\right|+\left|1-\frac{z}{|z|}\frac{w}{|w|}\right|\,,\,  1/2 \le |z|\le 2\,,\, 1/2\le |w|\le 2\,.
$$
Then it is easy to see that for all $z, w: |z|\le 1, |w|\le 1,$  we have
$$
|K_m(z,w)|\lesssim\frac{1}{\Delta(z,w)^{m}}.
$$
This holds because we know that
$$
\left| \frac{1}{(1-\bar{z}\cdot w)^m}\right|\lesssim\frac{1}{\Delta(z,w)^m}
$$
by \cites{B, T1, T2}.  Also in \cites{T1, T2} it is proved that if $\Delta(\zeta, w) <<\Delta(z, w)$ then
$$
|K_m(\zeta,w)-K_m(z,w)|\lesssim \frac{\Delta(\zeta, w)^{1/2}}{\Delta(z, w)^{m+1/2}}\,.
$$
This estimate then says that the kernel $K_m$ is a Calder\'on--Zygmund kernel with associated \cz\, parameter $\tau=1/2$ defined on the closed unit ball, but with respect to the quasi-metric $\Delta(z,w)$.  Finally, let $\mu$ be any probability measure compactly supported in the ball $\mathbb{B}_{2d}$.  

For the \cz\, kernel $K_m$ given above, we have the associated operator $T_{\mu,m}$ given by
$$
T_{\mu,m}(f)(x):=\int_{\mathbb{R}^{d}}K_m(x,y)f(y)d\mu(y)
$$
is a Calder\'on--Zygmund operator of order $m$ on the closed unit ball of $\mathbb{B}^d$ which satisfies the hypotheses of the Theorem \ref{md} on the kernel $K_m$ and measure $\mu$, however with respect to the quasi-metric $\Delta(x,y)$. If $\mu$ is compactly supported in the ball, the integral above converges absolutely. Otherwise we will consider the approximations of $\mu$ having supports compactly contained in the unit ball, and we will be interested in the uniform estimates of such operators. We must now argue that this  presents no problem in the application of Theorem \ref{md}.  

There is the following analog of Theorem \ref{md} pertinent to new metric corresponding to $\Delta$.

\begin{theorem}[Main Result 2]
\label{md1}
Let $k(z,w)$ be a \cz\, kernel of order $m\leq 2d$ on $X:=\{1/2\le |z|\le 2\}\subset \mathbb{C}^{d}$, with \cz\,  constants $C_{CZ}$ and $\tau$, but with respect to the metric $\Delta$ introduced above.  Let $\mu$ be a probability  measure with compact support in 
$X\cap \mathbb{B}_{2d}$, and suppose that all balls $B_{\Delta}$ in the metric $\Delta$ such that $\mu(B_{\Delta}(x,r)) >r^m$ lie in an open set $H$. Let also
$$
|k(z,w)|\le \frac1{\max (d(z)^m, d(w)^m)}\,,
$$
where $d(z):=\dist_{\Delta} (z, \mathbb{C}^{d}\setminus H)$.  Finally, suppose also that a ``$T1$ Condition'' holds for the operator $T$ with kernel $k$ and for the operator $T^*$ with kernel $k(w,z)$.  Namely, for all $\Delta$-cubes $Q$ we have:
\begin{equation}
\|T_{\mu,m}\chi_Q\|_{L^2(X;\mu)}^2 \le A\,\mu(Q)\,,\,\|T^*_{\mu,m}\chi_Q\|_{L^2(\R^d;\mu)}^2 \le A\,\mu(Q)\,.
\end{equation}
Then $\|T_{\mu,m}\|_{L^2(X;\mu)\rightarrow L^2(X:\mu)} \le C(A,m,d,\tau)$.
\end{theorem}

Theorems \ref{md} and \ref{md1} will be proved in the last section. However, now we wish to provide a few remarks concerning the $\Delta$-cubes $Q$ appearing in Theorem \ref{md1}, and indicate why we refer to them as ``cubes''.

In the proof of Theorem \ref{md} and Theorem \ref{md1} the cubes considered arise from a naturally constructed dyadic lattice.  In the standard Euclidean case, Theorem \ref{md}, we simply will take the standard dyadic lattice.  However, in the proof of Theorem \ref{md1} we will need to transfer certain parallelepiped regions to a spherical neighborhood of the sphere.  This situation arises since, in a neighborhood of the sphere, the metric $\Delta$ will look like a variant of the standard Euclidean metric but with different powers appearing.  It would be geometrically ``nicer'' if these shapes were actual (non-Euclidean) balls. However, this difficulty can be overcome since one actually can replace these non-Euclidean cubes by non-Euclidean balls.  This change between cubes and balls is \textit{not} a total triviality because the operator $T_{\mu,m}$ does not have a positive kernel.   In the Euclidean setting, we refer the reader to \cite{NTV1} for this passage from cubes to balls, and remark that these changes can be adapted to handle the case of a non-Euclidean metric verbatim.  At the end of the paper when we prove Theorem \ref{md1} we provide a few more words about the structure of the ``cubes'' $Q$ that appear.  


The next result is a corollary of Theorem \ref{md1}. We apply this corollary to obtain the characterization of Carleson measures in the whole scale of our spaces of analytic functions in the ball. Consider our special kernels
$$
K_{m}(z, w) =\textnormal{Re}\, \frac{1}{(1-\bar{w}\cdot z)^{m}}\,.
$$
Consider a probability measure $\mu$ supported on $X\subset \mathbb{B}_{2d}$. A well-known necessary condition for  the space $B_2^{\sigma}(\mathbb{B}_{2d})$
to be imbedded in $L^2(\mathbb{B}_{2d};\mu)$ is
\begin{equation}
\label{nec}
\mu(B_{\Delta}(\zeta, r)) \le C_1\, r^{2\sigma}\,,\, \forall \zeta \in \partial \mathbb{B}_{2d}\,.
\end{equation}
This is easily seen by testing the embedding condition on the reproducing kernel for the space of functions.  However, \eqref{nec} can be rewritten in a form akin to the conditions on the measure in Theorem \ref{md1}. Namely, of course \eqref{nec} is equivalent to (with another constant)
\begin{equation}
\label{nec1}
\mu(B_{\Delta}(\zeta, r)) \le C_2\, r^{2\sigma}\,,\, \forall B_{\Delta}(\zeta, r)):\, B_{\Delta}(\zeta, r))\cap \mathbb{C}^{d}\setminus \mathbb{B}_{2d}\neq \emptyset\,.
\end{equation}

In turn, \eqref{nec1} can be rephrased as saying
\begin{equation}
\label{nec2}
\text{every metric ball such that}\,\mu(B_{\Delta}(\zeta, r)) > C_2\, r^{2\sigma}\,,\,\text{is contained in the unit ball}\,\mathbb{B}_{2d}\,.
\end{equation}

\begin{theorem}[Main Result 3]
\label{md2}
 Let $\mu$ be a probability  measure with compact support in 
$X\cap \{z\in \mathbb{C}^{d}: 1/2\le |z|<1\}$ and all balls $B_{\Delta}$ in the metric $\Delta$ such that $\mu(B_{\Delta}(x,r)) >r^m$ lie inside the ball $\mathbb{B}_{2d}$.   Finally, suppose that for all $\Delta$-cubes $Q$ a ``$T1$ Condition'' holds for the operator $T_{\mu,m}$ with kernel $K_m$:
\begin{equation}
\|T_{\mu,m}\chi_Q\|_{L^2(X;\mu)}^2 \le A\,\mu(Q)\, \textnormal{ for all }.
\end{equation}
Then $\|T_{\mu,m}\|_{L^2(X;\mu)\rightarrow L^2(X;\mu)} \le C(A,m,d,\tau)$.
\end{theorem}

\begin{proof}
The proof is just a particular case of Theorem \ref{md1}.   We already noted that the kernels $K_m$ satisfy the estimate
$$
|K_m(z,w)|\le \frac1{\max (d(z)^m, d(w)^m)}\,,
$$
where $d(z):=\dist_{\Delta} (z, \mathbb{C}^{d}\setminus \mathbb{B}_{2d})$. In fact it is obvious, because with respect to the metric $\Delta$ (and not only in Euclidean metric) $d(z) = 1-|z|$.

On the other hand in \cites{T1, T2} it is shown that $K_m(z, w)$ on $\mathbb{B}_{2d}$ is a Calder\'on--Zygmund operator with parameters $m, \tau=1/2$ with respect to the quasi-metric $\Delta$.  We remark that this is a non-trivial observation.  All together this implies that we are under the assumptions of Theorem \ref{md1}, and we are done.
\end{proof}


Theorem \ref{md2} is, as we already noted, exactly the following theorem giving a necessary and sufficient condition for the ``Bergman--type'' operators to be bounded on $L^2(\mu)$.

\begin{theorem}
\label{Bergman}
Let $\mu$ be a probability measure  supported in $\{z\in \mathbb{C}^{d}: 1/2\le |z|<1\}$.  Then the following assertions are equivalent:
\begin{itemize}
\item[(a)] The measure $\mu$ satisfies the following conditions:
\begin{itemize}
\item[(i)] $\mu(B_{\Delta}(x,r)) \le C_1 \, r^m\,,\,\,\forall B_{\Delta}(x,r): B_{\Delta}(x,r)\cap \mathbb{C}^{d}\setminus\mathbb{B}_{2d}\neq\emptyset;$
\item[(ii)] For all $\Delta$-cubes $Q$ we have $\| T_{\mu,m} \chi_Q\|_{L^2(X;\mu)}^2\leq C_2\mu(Q)$. 
\end{itemize} 
\item[(b)]
$$
\|T_{\mu,m}\|_{L^2(X;\mu)}\le C_3<\infty\,.
$$ 
\end{itemize}
Here $C_3=C(C_1, C_2, m, \tau)$, $C_1=C(C_3, m)$, $C_2=C(C_3,m)$.
\end{theorem}

We next apply this Theorem to study the Carleson measures for the Besov--Sobolev spaces of analytic functions on the unit ball in $\mathbb{C}^d$.

\subsection{Carleson Measures on Besov--Sobolev Spaces of Analytic Functions}
\label{Carleson}

In this section we give a characterization of Carleson measures for the Besov--Sobolev spaces $B_\sigma^2(\mathbb{B}_{2d})$.  

Recall that the space $B_\sigma^2(\mathbb{B}_{2d})$ is the collection of analytic functions on the unit ball $\mathbb{B}_{2d}$ in $\mathbb{C}^d$ and such that for any integer $m\geq 0$ and any $0\leq\sigma<\infty$ such that $m+\sigma>\frac{d}{2}$ we have the following norm being finite:
$$
\|f\|_{B^\sigma_2}^2:=\sum_{j=0}^{m-1}|f^{(j)}(0)|^2+\int_{\mathbb{B}_{2d}}|(1-|z|^2)^{m+\sigma}f^{(m)}(z)|^2\frac{d\,V(z)}{(1-|z|^2)^{d+1}}.
$$

One can show that these spaces are independent of $m$ and are Hilbert spaces, with obvious inner products.  The spaces $B_2^\sigma(\mathbb{B}_{2d})$ are reproducing kernel Hilbert spaces with kernels given by $k_\lambda^\sigma(z)=\frac{1}{\left(1-\overline{\lambda}\cdot z\right)^{2\sigma}}$.  A minor modification has to be made when $\sigma=0$, but this introduces a logarithmic reproducing kernel.

A non-negative measure $\mu$ supported inside $\mathbb{B}_{2d}$ is called a $B_\sigma^2(\mathbb{B}_{2d})$-Carleson measure if
$$
\int_{\mathbb{B}_{2d}}|f(z)|^2d\mu(z)\leq C(\mu)^2\|f\|_{B_\sigma^2(\mathbb{B}_{2d})}^2\quad\forall f\in B_\sigma^2(\mathbb{B}_{2d}).
$$
This is a function theoretic property and is looking for the measures $\mu$ that ensure the continuous embedding of $B_\sigma^2(\mathbb{B}_{2d})\subset L^2(\mathbb{B}_{2d};\mu)$.  

The norm of the Carleson measure $\mu$ is given by the best constant possible in the above embedding.  There are also geometric ways that one can characterize the $B_2^\sigma(\mathbb{B}_{2d})$-Carleson measures.  These characterizations are typically given in terms of the ``capacity'' associated to the function space $B_2^\sigma(\mathbb{B}_{2d})$ and an interaction between the geometry of certain sets arising from the reproducing kernel $k_\lambda^\sigma(z)$.  See for example \cites{CO}.  However, these characterizations had the restriction of only working in the range $0\leq\sigma\leq\frac{1}{2}$, and when $\frac{d}{2}\leq\sigma$.  Namely, previous methods were unable to answer the question in the difficult range of $\frac{1}{2}<\sigma<\frac{d}{2}$.  However, using the difficult methods of non-homogeneous harmonic analysis, we can give a characterization of the $B_2^\sigma(\mathbb{B}_{2d})$ using the Main Theorem in the form of Theorem \ref{Bergman} (or, equivalently, Theorem \ref{md2}) for all values of $\sigma$ at once.

\medskip

\begin{remark}
We can always assume that the support of the measure is {\it compactly} contained in the ball $\mathbb{B}_{2d}$. In fact, in this case the embedding is always continuous, but we are interested in the {\it norm} of this embedding, and its effective estimate. This assumption allows us to not worry about the boundedness per se, which is pleasant when working with Calder\'on--Zygmund kernels and excludes questions of absolute convergence of integrals. This also explains why we needed Theorems \ref{md2} and \ref{Bergman} only for compactly supported measures.
\end{remark}

\medskip

We begin by recalling the following fact found in the paper by N. Arcozzi, R. Rochberg, and E. Sawyer, \cite{ARS}.  First, some notation since the following proposition holds in an arbitrary Hilbert space with a reproducing kernel.  Let $\mathcal{J}$ be a Hilbert space of functions on a domain $X$ with reproducing kernel function $j_x$.  In this context, a measure $\mu$ is Carleson exactly if the inclusion map $\iota$ from $\mathcal{J}$ to $L^2(X;\mu)$ is bounded.

\begin{prop}
\label{Carl}
A measure $\mu$ is a $\mathcal{J}$-Carleson measure if and only if the linear map
$$
f(z)\to T(f)(z)=\int_{X}\textnormal{Re}\,j_x(z)f(x)d\mu(x)
$$
is bounded on $L^2(X;\mu)$.
\end{prop}

\begin{proof}
The inclusion map $\iota$ is bounded from $\mathcal{J}$ to $L^2(X;\mu)$ if and only if the adjoint map $\iota^*$ is bounded from $L^2(X;\mu)$ to $\mathcal{J}$, namely,
$$
\|\iota^* f\|^2_{\mathcal{J}}=\langle \iota^* f,\iota^*f\rangle_{\mathcal{J}}\leq C\|f\|_{L^2(X;\mu)}^2, \quad\forall f\in L^2(X;\mu).
$$
For an $x\in X$ we have 
\begin{eqnarray*}
\iota^* f(x)=( \iota^*f, j_x)_\mathcal{J} & = & ( f, \iota j_x)_{L^2(X;\mu)}\\
 & = & \int_X f(w)\overline{j_x(w)}d\mu(w)\\
 & = & \int_{X}f(w) j_w(x)d\mu(w)\\
\end{eqnarray*}
and using this computation, we obtain that
\begin{eqnarray*}
\|\iota^* f\|^2_\mathcal{J} & = & (\iota^* f,\iota^*f)_{\mathcal{J}}\\
 & = & \left(\int_X j_wf(w)d\mu(w), \int_X j_{w'}f(w')d\mu(w') \right)_{\mathcal{J}}\\
 & = & \int_X\int_X ( j_w, j_{w'})_{\mathcal{J}}f(w)d\mu(w)\overline{f(w')}d\mu(w')\\
 & = & \int_X\int_X j_w(w')f(w)d\mu(w)\overline{f(w')}d\mu(w').
\end{eqnarray*}
Since we know that the continuity of $\iota^*$ for general $f$ is equivalent to having it for real $f$, without loss of generality, we can suppose that $f$ is real.  In that case, we can continue the estimates with
$$
\|\iota^* f\|^2_{\mathcal{J}}=\int_X\int_X \textnormal{Re}\, j_w(w') f(w)f(w')d\mu(w)d\mu(w')=(Tf, f)_{L^2(X;\mu)}.
$$ 
But, the last quantity satisfies the required estimates exactly when the operator $T$ is bounded.
\end{proof}

When we apply this proposition to the spaces $B_2^\sigma(\mathbb{B}_{2d})$ this suggests that we study the operator
$$
T_{\mu,2\sigma}(f)(z)=\int_{\mathbb{B}_{2d}} \textnormal{Re}\left(\frac{1}{(1-\overline{w}z)^{2\sigma}}\right) f(w)d\mu(w).
$$
And, it is here that the tools of non-homogeneous harmonic analysis will play a role.  Using Theorem \ref{Bergman}, we have the following characterization of Carleson measures for $B_2^\sigma(\mathbb{B}_{2d})$.  The statement below is simply a restatement of Theorem \ref{Carl2}.

\begin{theorem}
Suppose that $0<\sigma$.  Let $\mu$ be a positive Borel measure in $\mathbb{B}_{2d}$.  Then the following conditions are equivalent:
\begin{itemize}
\item[(a)] $\mu$ is a $B_2^\sigma(\mathbb{B}_{2d})$-Carleson measure;
\item[(b)] $T_{\mu,2\sigma}:L^2(\mathbb{B}_{2d};\mu)\to L^2(\mathbb{B}_{2d};\mu)$ is bounded;
\item[(c)] There is a constant $C$ such that

\begin{itemize}
\item[(i)] $\|T_{\mu,2\sigma}\chi_Q\|_{L^2(\mu)}^2 \le C\,\mu(Q)$ for all $\Delta$-cubes $Q$;
\item[(ii)] $\mu(B_{\Delta}(x,r))\leq C\,r^{2\sigma}$ for all balls $B_{\Delta}(x,r)$ that intersect $\mathbb{C}^{d}\setminus\mathbb{B}_{2d}$.
\end{itemize}
\end{itemize}
\end{theorem}

The equivalence between $(a)$ and $(b)$ follows simply from the abstract Proposition \ref{Carl}.  The equivalence between $(b)$ and $(c)$ follows from Theorem \ref{Bergman} and a well-known remark that (ii) is necessary for (b) to hold, see e.g. \cites{ARS, T1, T2}.


We remark that in the range $0\leq\sigma\leq\frac{1}{2}$ Theorem \ref{Carl2} was known.  It was proved by N. Arcozzi, R. Rochberg, and E. Sawyer using the machinery of function spaces on trees, see \cite{ARS}.  It was also proved, more in the spirit of what appears in this paper,  by E. Tchoundja, see \cites{T1, T2}.  He adapted the proof of J. Verdera of the boundedness of the Cauchy operator on $L^2(\mu)$, \cite{Verdera}.  Key to both these approaches was the fact that when $0\leq\sigma\leq\frac{1}{2}$ a certain ``positivity'' of the kernel could be exploited.  This positivity is cruelly missing when $\sigma>\frac{1}{2}$, but can be overcome via the ``$T1$ Conditions''.  We also remark that the range $\frac{1}{2}<\sigma<\frac{d}{2}$, was previously unknown, and this theorem then answers the question in the whole range of possible $\sigma$.

\subsection{Theorem \ref{Bergman} and the Standard Carleson Measure Conditions}
\label{standard}

In this sub-section, we show that the conditions in Theorem \ref{Bergman} reduce to the well known geometric conditions for Carleson measures on the Hardy space and the Bergman space.  Setting $\sigma=\frac{n}{2}$ one can show that the space $B_2^{\frac{n}{2}}(\mathbb{B}_{2n})=H^2(\mathbb{B}_{2n})$, the Hardy space of analytic functions on the unit ball $\mathbb{B}_n$.  Similarly, if we set $\sigma=\frac{n+1}{2}$, the we see that $B_2^{\frac{n+1}{2}}(\mathbb{B}_n)=A^2(\B_n)$, the space of square integrable functions on $\mathbb{B}_n$ with respect to Lebesgue measure.

We further focus on the case in one dimension, $d=1$, since the ideas are easier to see, though they readily can be adapted to several variables and could be applied to demonstrate that the known geometric conditions for Carleson measures in the Hardy space of the ball and Bergman space of the ball are equivalent to a the more exotic T1 condition appearing in Theorem \ref{Bergman}.  However, when the geometric condition is more difficult to understand, such as when it is characterized via capacity conditions, then it is not immediately clear that the geometric and the T1 conditions are in fact equivalent without passing through the equivalence of the main theorem. 

Finally, we will work in the case of the upper half plane $\mathbb{C}_+$.  The main theorems apply in this context, either via conformal mappings and changes of variables.  Key to these computations is the M\"obius invariance of the function spaces $B_2^{\sigma}(\mathbb{B}_{2d})$.  Or, one can set-up the entire problem in the context of the upper-half plane, then simply apply the method of proof of the main theorems to obtain the geometric characterization in this context.

\subsubsection{Carleson Measures on the Hardy Space}

A measure $\mu$ is a Carleson measure for $H^2(\mathbb{C}_+)$ (or $H^2(\mathbb{D})$) if 
$$
\int_{\mathbb{C}_+}\vert f(z)\vert ^2d\mu(z)\leq C(\mu)^2\| f\|_{H^2(\mathbb{C}_+)}^2\quad \forall f\in H^2(\mathbb{C}_+).
$$ 

As is well-known this function-theoretic condition happens if and only if the follow geometric condition is satisfied for all tents $T(I)$ over $I\subset\mathbb{R}$:
$$
\mu\left(T(I)\right)\leq C \vert I\vert \quad\forall I\subset\mathbb{R}.
$$

First note that if $I\subset\R$ is an interval, then $T(I)$ will be a cube in $\R^2$.  If we know that Condition (ii)  of Theorem \ref{Bergman}  (in the case when $d=1$ and $\sigma=\frac{1}{2}$) \ref{Bergman} holds, then restricting the outer integral in that condition to $T(I)$ and using standard estimates for the kernel one can see that
$$
\frac{\mu(T(I))^3}{\vert I\vert^2}\leq C(\mu)\mu(T(I))
$$
which upon rearrangement, yields, the geometric Carleson condition.

Next, we want to show that if $\mu$ satisfies the geometric Carleson condition, then in fact we have Condition (ii)  of Theorem \ref{Bergman}  holding as well. In other words, we want to see easily why the usual geometric Carleson condition on $\mu$ implies the boundedness of our operator on characteristic functions of squares. 

We first observe that the function
$$
F_{Q,\mu}(z):=\int_{\R^2} \frac{\chi_Q(\xi)}{\xi-z}d\mu(\xi)
$$
belongs to $H^2(\mathbb{C}_-)$ (since $\mu$ is supported on the upper-half plane) with norm a constant multiple of $\sqrt{\mu(Q)}$.  This is a well known fact and can be found for example in \cite{NTV4}, however, we provide the proof now for convenience.  Let $\varphi$ be a test function on $\mathbb{R}$, then we have that 
$$
\Phi(z):=\int_{\mathbb{R}}\frac{\varphi(x)}{x-z}dx
$$
is analytic and belongs to the Hardy class $H^2(\mathbb{C}_-)$.  Next, note that, 
\begin{eqnarray*}
\int_\mathbb{R} \varphi(x) F_{Q,\mu}(x)dx & = & \int_{Q} \Phi(\xi) d\mu(\xi).
\end{eqnarray*}

Using H\"older's inequality and using that $\mu$ is a Carleson measure, and so the function theoretic embedding also holds, we see that for all test functions we have
\begin{eqnarray*}
\left\vert \int_\mathbb{R} \varphi(x) F_{Q,\mu}(x)dx\right\vert  & = & \left\vert \int_{Q} \Phi(\xi) d\mu(\xi)\right\vert \\
& \leq &   \sqrt{\mu(Q)}\left(\int_{\mathbb{C}_-}\vert\Phi(\xi)\vert^2 d\mu(\xi)\right)^{1/2}\\
& \leq & \| \varphi\|_{L^2(\R^d;\mu)}\sqrt{\mu(Q)}.
\end{eqnarray*}

Taking the supremum over all possible test functions proves the claim.  However, since the function $F_{Q,\mu}$ belongs to $H^2(\mathbb{C}_-)$, then another application of the Carleson Embedding property (or equivalently, the geometric condition for Carleson measures) gives that
$$
\| T_{\mu,\frac{1}{2}} \chi_Q\|_{L^2(\mu)}^2 = \int_{\mathbb{C}_-}\vert F_{Q,\mu}(z)\vert^2 d\mu(z)\leq C(\mu) \mu(Q).
$$

Thus, we see that in this case, we have the equivalence between the standard Carleson Embedding Condition and Condition (ii)  of Theorem \ref{Bergman}.

\subsubsection{Carleson Measures on the Bergman Space}

A measure $\mu$ is a Carleson measure for the Bergman space $A^2(\mathbb{C}_+)$ if we have that
$$
\int_{\mathbb{C}_+}|f(z)|^2d\mu(z)\leq C(\mu)^2\| f\|_{A^2(\mathbb{C}_+)}^2\quad\forall f\in A^2(\mathbb{C}_+).
$$

Let $F_{\mathbb{C}_+}$ denote the set of squares $Q$ in the upper half-plane with sides parallel to the axis such that the boundary of $Q$ intersects $\R$.  A well-known fact is that this function theoretic condition is equivalent to the following geometric condition
$$
\mu(Q)\leq C|Q|\quad\forall Q\in F_{\mathbb{C}_+}.
$$
This can be found in the paper by Hastings, \cite{Hastings} (after a conformal change of variable).  See also the paper by Luecking, \cite{Lu}.  We want to show that the ``T1 Condition'' that appears in Theorem \ref{Bergman} reduces to this well known geometric condition.

Again, by standard estimates for the \cz\, kernel in the case when $d=1$, $\sigma=1$, we have the well-known geometric characterization of the Carleson measures for the Bergman space 
\begin{equation}
\label{everyQ}
\mu(Q)\leq C(\mu)|Q|\quad\forall Q\in F_{\mathbb{C}_+}.
\end{equation}

We now want to demonstrate that if the standard geometric characterization \eqref{everyQ} of Carleson measures for the Bergman spaces holds, then we have that Condition (ii) of Theorem \ref{Bergman} holds as well.  
Denote 
$$
t_{\mu}(Q,z):= \int_{\mathbb{C}_+}\frac{1}{(\zeta-\bar{z})^2}\,\chi_Q(\zeta) \,d\mu(\zeta).
$$
We want to deduce from \eqref{everyQ} that
\begin{equation}
\label{Bt}
\|t_{\mu}(Q,z)\|_{L^2(\mathbb{C}_+;\mu)}^2 \le C\,\mu(Q)
\end{equation}
for any $Q\in F_{\mathbb{C}_+}$.  The idea behind the proof will be to decompose the two appearances of the measure $\mu$ in \eqref{Bt} with simpler analogous that arise from condition \eqref{everyQ}.  These decompositions will introduce error terms that need to be controlled, however these errors will be positive operators that can be controlled via Schur's test.  We now explain in detail.

Fix $ Q\in F_{\mathbb{C}_+}$ and decompose it to the union of rectangles $q$ such that $\dist (q, \R) \approx \diam (q)$.
This can be done in a standard fashion. Now define a function $w$ that is constant on each $q$ with value  given by $\mu(q)/|q|$.
Then by the geometric definition of Carleson measures on the Bergman space we have that
\begin{equation}
\label{w}
\|w\|_{L^{\infty}(\mathbb{C}_+)} \le A\, C(\mu)\,,
\end{equation}
where $A$ is an absolute constant, and $C(\mu)$ is the Carleson constant appearing \eqref{everyQ}. Also
\begin{equation}
\label{w1}
\int_Q w\, dm_2=\mu(Q).
\end{equation}

We write $t_{\mu}(Q,z)$ as the sum of integrals over all the $q$'s. If we replace the measure $\mu$ in $q$ by $w\,dm_2$ in $q$ we will get an error, which is easy to estimate, namely
\begin{equation}
\label{err}
|t_{\mu}(Q,z) - t_{w\,dm_2}(Q,z)| \le A\,\int_{\mathbb{C}_+}\frac{|\Im\zeta|}{|\zeta-\bar{z}|^3}\,\chi_Q(\zeta) \,w\,dm_2(\zeta).
\end{equation}

Denote 
$$
K(\zeta,z):=\frac{|\Im\zeta|}{|\zeta-\bar{z}|^3}\,.
$$

To prove \eqref{Bt}, by an application of \eqref{err}, we are left with proving the following two estimates
\begin{eqnarray}
\label{Bt1}
\|t_{w\,dm_2}(Q,z)\|_{L^2(\mathbb{C}_+;\mu)}^2 & \le & C\,\mu(Q)\quad\forall Q\in F_{\mathbb{C}_+};\\ 
\label{Bt2}
\left\Vert\int K(\zeta,z)\,\chi_Q(\zeta) \,w\,dm_2(\zeta)\right\Vert_{L^2(\mathbb{C}_+;\mu)}^2  & \le &  C\,\mu(Q)\quad\forall Q\in F_{\mathbb{C}_+}.
\end{eqnarray}

We first focus on proving \eqref{Bt1}.  Repeating the averaging procedure described above, we can replace the measure $\mu$ by $wdm_2$ and \eqref{Bt1} reduces to proving
\begin{equation}
\label{Bt3}
\|t_{w\,dm_2}(Q,z)\|_{L^2(\mathbb{C}_+;w\,dm_2)}^2 \le C\,\mu(Q)\quad\forall Q\in F_{\mathbb{C}_+}.
\end{equation}
Again, performing this averaging introduces an error term which is at most proportional (up to an absolute constant) to
\begin{equation}
\label{Bt4}
\left\Vert\int K_2(\zeta,z)\,\chi_Q(\zeta) \,w\,dm_2(\zeta)\right\Vert_{L^2(\mathbb{C}_+;\mu)}\,,
\end{equation}
with
$$
K_2(\zeta,z):=\frac{|\Im z|}{|\zeta-\bar{z}|^3}\,.
$$

Now \eqref{Bt3} is obvious from \eqref{w} and from the fact that $1/(\zeta-z)^2$ represents the kernel of a Calder\'on-Zygmund operator on the plane. In particular, it is bounded on $L^2(\mathbb{C}_+;m_2)$, and using \eqref{w} and \eqref{w1} we obtain
$$
\|t_{w\,dm_2}(Q,z)\|_{L^2(\mathbb{C}_+;w\,dm_2)}^2 \le A\, \|w\|_{L^{\infty}(\mathbb{C}_+)}\, \int_Q w^2\, dm_2 \le  A\, \|w\|_{L^{\infty}(\mathbb{C}_+)}^2\, \int_Q w\, dm_2=A\, C(\mu)^2 \,\mu(Q)\,,
$$
which is exactly \eqref{Bt3}. 

We are left to prove \eqref{Bt2} and \eqref{Bt4}. But, these kernels are positive, and they change by an absolute multiplicative constant when the argument runs over any $q$ which we introduced in the averaging of the measure $\mu$. So \eqref{Bt2} is the same as the following estimate:
\begin{equation}
\label{Bt5}
\left\Vert\int K_1(\zeta,z)\,\chi_Q(\zeta) \,w\,dm_2(\zeta)\right\Vert_{L^2(\mathbb{C}_+;w\,dm_2)}^2 \le C\,\mu(Q)\quad\forall Q\in F_{\mathbb{C}_+}.
\end{equation}
We are left to prove \eqref{Bt5} and its counterpart with $K_2$ appearing.  However, one can easily see that these positive kernels induce bounded operators $L^2(\mathbb{C}_+;m_2)\rightarrow L^2(\mathbb{C}_+; m_2)$. 

A direct calculation shows that:
$$
\int_{\mathbb{C}_+} \frac{|\Im \zeta|}{|\Im z|^{1/2}|\zeta-\bar{z}|^3}\,dm_2(z) \le A\, \frac{1}{|\Im\zeta|^{1/2}}\,.
$$
$$
\int_{\mathbb{C}_+} \frac{|\Im \zeta|^{1/2}}{|\zeta-\bar{z}|^3}\,dm_2(\zeta) \le A\, \frac{1}{|\Im z|^{1/2}}\,.
$$
Since they are positive operators, one can conclude they are bounded operators by an application of Schur's Test and the computations above.  So, 
\begin{equation}
\label{Bt6}
\left\Vert\int K_1(\zeta,z)\,\chi_Q(\zeta) \,w\,dm_2(\zeta)\right\Vert_{L^2(\mathbb{C}_+;w\,dm_2)}^2 \le A\, \|w\|_{L^{\infty}(\mathbb{C}_+)} \,\int_Q w^2\, dm_2\le 
A\, C(\mu)^2\, \mu(Q)\,,
\end{equation}
by \eqref{w} and \eqref{w1} again. But this is exactly \eqref{Bt5}. The same argument holds for the kernel $K_2$.

\bigskip

\section{The beginning of the proofs of Theorems \ref{md}, \ref{md1}. Littlewood--Paley Decompositions with respect to random lattices:  Terminal and Transit Cubes}

The difficulty lies in the fact that $\mu$ is virtually arbitrary. For example, a general measure $\mu$ need not have the doubling property. The key idea about how to cope with the ``badness'' of $\mu$ is the use of randomness. We will decompose functions using random dyadic lattices, and then resulting averages will somehow be good.  This will allow us to ``smooth'' out the difficulties of the measure $\mu$.

\medskip

The proof of Theorem \ref{md} will be divided in several parts.  We begin by collecting collections of cubes that are beneficial for the proof of the Theorem.  These collections are then used to decompose arbitrary functions in $L^2(\R^d;\mu)$.  In the remaining sections since there is no potential for confusion, we will let $T$ denote the operator $T_{\mu,m}$ appearing in Theorem \ref{md}.

We assume that $F=\supp\mu$ lies in a  cube $\frac14 Q$, where $Q$ is a certain unit cube. With out loss of generality, consider this cube as one of the unit cubes of a standard dyadic lattice $\Dk$.  Furthermore, let $\Dk_1$, $\Dk_2$ be two shifted dyadic lattices, one by a small shift $\om_1$, another by $\om_2$, where $\om_i \in \frac1{20} Q=\Omega$. These shifts are independent and the measure $\mathbb{P}$ on $\Omega$ is normalized to $1$.

\subsection{Terminal and transit cubes}
\label{tt}

We will call the cube $Q\in \Dk_i$ a {\it terminal cube} if $2Q$ is contained in our open set $H$ or $\mu(Q)=0$.  All other cubes are called {\it transit}.   Let us denote by $\Dk_i^{term}$ and $\Dk_i^{tr}$ the terminal and transit cubes from $\Dk_i$.   We first state two obvious Lemmas.

\begin{lemma}
\label{obv1}
If $Q$ belongs to $\Dk_i^{term}$, then
$$
|k(x,y)|\le \frac1{\ell(Q)^m}\,.
$$
\end{lemma}
This follows since $Q\subset 2Q\subset H$ and so for $x,y\in Q$ we have that $d(x)\geq\ell(Q)$ and similarly for $y$.  Another obvious lemma:

\begin{lemma}
\label{obv2}
If $Q$ belongs to $\Dk_i^{tr}$, then
$$
\mu(Q) \le C(d)\,\ell(Q)^m\,.
$$
\end{lemma}

\medskip

We would like to denote $Q_j$ as a dyadic cube belonging to the dyadic lattice $\Dk_j$.  Unfortunately, this makes the notation later very cumbersome.  So, we will use the letter $Q$ to denote a dyadic cube belonging to the lattice $\Dk_1$ and the letter $R$ to denote a dyadic cube belonging to the lattice $\Dk_2$.

From now on, we will always denote by $Q_j$ ($j=1,\dots ,2^d$) the $2^d$ dyadic subcubes of a cube $Q$ enumerated in some ``natural order''.   Similarly, we will always denote by $R_j$ ($j=1,\dots ,2^d$) the $2^d$ dyadic subcubes of a cube $R$ from $\mathcal{D}_2$.  

Next, notice that there are special unit cubes $Q^0$ and $R^0$ of the dyadic lattices $\Dk_1$ and $\Dk_2$ respectively.   They have the property that they are both transit cubes and contain $F$ deep inside them.

\subsection{The Projections $\Lambda$ and $\Delta_Q$}
\label{pr}

Let $\mathcal{D}$ be one of the dyadic lattices above. For a function $\psi\in L^1(\mathbb{R}^d;\mu)$ and for a cube $Q\subset\R^d$, denote by
$\langle \psi\rangle\ci Q$ the average value of $\psi$ over $Q$ with respect to the measure $\mu$, i.e.,

$$
\langle \psi\rangle\ci Q := \frac{1}{\mu(Q)}\int_Q\psi\,d\mu
$$
(of course, $\langle \psi\rangle\ci Q$ makes sense only for cubes $Q$ with $\mu(Q)>0$).
Put
$$
\Lambda\vf :=\langle \vf\rangle\ci {Q^0}\,.
$$
Clearly, $\Lambda\vf\in L^2(\mathbb{R}^d;\mu)$ for all $\vf\in L^2(\mathbb{R}^d;\mu)$, and $\Lambda^2=\Lambda$, i.e., $\Lambda$ is a projection. Note also, that actually $\Lambda$
does not depend on the lattice $\mathcal{D}$ because the average is taken over the whole support of the measure $\mu$ regardless of the position of the cube $Q^0$ (or $R^0$).

\medskip

\begin{remark}
Below we will start almost every claim by ``Assume (for definiteness) that $\ell(Q)\le \ell(R)$\dots ''.  Later, for ease of notation, we will write that a cube $Q\in\mathcal{X}\cap\mathcal{Y}$ and mean that the dyadic cube $Q$ has both property $\mathcal{X}$ and $\mathcal{Y}$ simultaneously.  
\end{remark}

\medskip

For every transit cube $Q\in\Dk_1$, define $\Delta\ci Q\vf$ by
$$
\Delta\ci Q\vf\bigr|\ci{\R^d\setminus (Q\cup X)}:=0,\qquad\,\,\,\,\,
\Delta\ci Q\vf\bigr|\ci{Q_j\cap X}:=\left\{\aligned\left[\langle \vf\rangle\ci {Q_j} -\langle \vf\rangle\ci {Q}\right] &\text{\quad if $Q_j$ is transit;}\\ \vf-\langle \vf\rangle\ci {Q} &\text{\quad if $Q_j$ is terminal}\endaligned\right.$$
($j=1,\dots , 2^d$).
Observe that for every transit cube $Q$, we have $\mu(Q)>0$, so our
definition makes sense since zero can not appear in the denominator.  We repeat the same definition for $R\in \mathcal{D}_2$.

We have the following easy properties of the operators $\Delta\ci Q\vf$.

\begin{lemma}
\label{DeltaQ}
For every $\vf\in L^2(\mathbb{R}^d;\mu)$ and every transit cube $Q$ the following properties hold:
\begin{itemize}
\item[(1)] \emph{$\Delta\ci Q\vf\in L^2(\mathbb{R}^d;\mu)$};
\item[(2)] \emph{$\int_{\R^d}\Delta\ci Q\vf\,d\mu=0$};
\item[(3)] \emph{$\Delta\ci Q$ is a projection, i.e., $\Delta\ci Q^2=\Delta\ci Q$};
\item[(4)] \emph{$\Delta\ci Q\Lambda=\Lambda\Delta\ci Q=0$};
\item[(5)]  \emph{If $Q,\widetilde{Q}$ are transit, 
$\widetilde{Q}\neq Q$, then $\Delta\ci Q\Delta\ci {\widetilde{Q}}=0$}.
\end{itemize}
\end{lemma}
The proof of this Lemma is given by direct computation and is left as an easy exercise for the reader.  Using the operators $\Delta\ci Q\vf$ we have a decomposition of $L^2(\mathbb{R}^d;\mu)$ functions.

\begin{lemma}
\label{Riesz}
For every $\vf\in L^2(\R^d;\mu)$ we have
$$
\vf=\Lambda\vf+\sum_{Q\in\Dk_1^{tr}}\Delta\ci Q\vf,
$$
the series converges in $L^2(\R^d;\mu)$ and, moreover,
$$
\|\vf\|^2\ci{L^2(\R^d;\mu)}=
\|\Lambda\vf\|^2\ci{L^2(\R^d;\mu)}+\sum_{Q\in\Dk_1^{tr}}\|\Delta\ci Q\vf\|^2\ci{L^2(\R^d;\mu)}\,.
$$
\end{lemma}

\begin{proof}
Note first of all that if one understands the sum 
$$\sum_{Q\in\Dk_1^{tr}}$$
as $\lim_{k\to\infty}\sum_{Q\in\Dk_1^{tr}:\ell(Q)> 2^{-k}}$, then for
$\mu$-almost every $x\in\R^d$, one has
$$
\vf(x)=\Lambda\vf(x)+\sum_{Q\in\Dk_1^{tr}}\Delta\ci Q\vf( x).
$$
Indeed, the claim is obvious if the point $x$ lies in some terminal cube.
Suppose now that this is not the case. Observe that

$$
\Lambda\vf(x)+\sum_{Q\in\Dk_1^{tr}:\ell(Q)> 2^{-k}}\Delta\ci Q\vf(x)=
\langle \vf\rangle\ci{Q^{k}} ,
$$
where $Q^{k}$ is the dyadic cube of size $2^{-k}$, containing $x$.
Therefore, the claim is true if
$$
\langle \vf\rangle\ci{Q^{k}}\to \vf(x)\,.
$$
But, the exceptional set for this condition has $\mu$-measure
$0$.

Now the orthogonality of all $\Delta_Q\vf$ between themselves, and their orthogonality to $\La\vf$ proves the lemma.

\end{proof}

\section{Good and bad functions}
\label{gb}

In this section we introduce a decomposition of functions $f\in L^2(\R^d;\mu)$ into ``good'' and ``bad'' parts based on random dyadic lattices.  Additionally, it is shown that the proof of Theorem \ref{md} can be reduced to the case of ``good'' functions, since the ``bad'' terms can be averaged out.  Finally, we decompose the needed estimates on the ``good'' functions into three distinct terms that will be handled separately.

We consider functions $f$ and $g \in L^2(\R^d;\mu)$. Fixing $\om_1,\om_2 \in \Omega$ we construct two dyadic lattices $\Dk_1$ and $\Dk_2$ as before.   By Lemma \ref{Riesz} the functions $f$ and $g$ have decompositions given by 
$$
f= \La f +\sum_{Q\in \Dk_1^{tr}} \Delta_Q f,\quad g=\La g +\sum_{R\in \Dk_2^{tr}} \Delta_R g.
$$

\bigskip

 For a  dyadic cube $R$ we denote $\cup_{i=1}^{2^d} \pd R_i$ by $sk \,R$, called the \textit{skeleton} of $R$. Here the $R_i$ are the dyadic children of $R$. 
 
\begin{defi}
Let $\tau, m$ be the parameters of the \cz \,kernel $k$. We fix $\alpha= \frac{\tau}{2\tau+2m}$ and a small number $\delta>0$. Fix $S\ge 2$ to be chosen later.  Choose an integer $r$ such that 
\begin{equation}
\label{r}
2^{-r}\le \delta^S < 2^{-r+1}\,.
\end{equation}
A cube $Q\in \Dk_1$ is called {\it bad} (or $\delta$-bad) if there exists $R\in \Dk_2$ such that

1) $\ell(R)\ge 2^r \ell(Q)\,,$

2) $\dist (Q, sk\,R) <\ell(Q)^{\alpha} \ell(R)^{1-\alpha}\,.$

Let $\mathcal{B}_1$ denote the collection of all bad cubes and correspondingly let $\mathcal{G}_1$ denote the collection of good cubes.  The analogous definitions give the collection of {\it bad} cubes $R\in \Dk_2$, $\mathcal{B}_2$ and the collection of good cubes $\mathcal{G}_2$.  
\end{defi}

\bigskip

We say, that $\vf = \sum_{Q\in \Dk_1^{tr}} \Delta_Q\vf$ is \textit{bad} if in the sum only bad $Q$'s participate in this decomposition. The same applies to 
$\psi = \sum_{Q\in \Dk_2^{tr}} \Delta_Q\psi$. In particular, given $\om_1,\om_2 \in \Omega$, we fix
the decomposition of $f$ and $g$ into good and bad parts:
$$
f= f_{good} + f_{bad}\,,\,\,\text{where}\,\, f_{good} = \La f + \sum_{Q\in \Dk_1^{tr}\cap\mathcal{G}_1} \Delta_Q f\,,
$$
$$
g= g_{good} + g_{bad}\,,\,\,\text{where}\,\, g_{good} = \La g + \sum_{R\in \Dk_2^{tr}\cap\mathcal{G}_2} \Delta_R g\,.
$$

\begin{theorem}
\label{badprob}
One can choose $S= S(\alpha)$ in such a way that for any fixed $Q\in \mathcal{D}_1$,
\begin{equation}\label{prob}\mathbb{P}_{\om_2}\{Q \,\text{is bad}\} \leq \delta^2\,.
\end{equation}
By symmetry $\mathbb{P}_{\om_1}\{R \,\text{is bad}\} \leq \delta^2$ for any fixed $R\in \mathcal{D}_2$.
\end{theorem} 
\begin{proof}
Consider the unit  cube $Q^0$ of $\mathcal{D}_1$, which satisfies $F\subset \frac14 Q^0$. In particular, $\om + Q^0$ contains $F$ in $\frac12 (\om + Q^0)$ for every $\om \in (-1/40,1/40]^d$. Recall that the normalized Lebesgue measure on the cube $(-1/40,1/40]^d$ is our probability measure. In the construction of our random lattices, all the lattices ``start" with the cube of unit side given by $\om + Q^0$, and then $\om +Q^0$ is dyadically subdivided, and thus generates the dyadic lattice $\mathcal{D}(\om)$.   We do two such constructions, giving us two such arbitrary lattices $\mathcal{D}_1$, $\mathcal{D}_2$.  By the definition above, a cube $Q\in\mathcal{D}_1$ is bad if there exists a cube $R\in\mathcal{D}_2$ such that $\dist(Q, \partial R)\le \ell(Q)^{\al}\ell(R)^{1-\al}$ and $\ell(R)\ge 2^r\ell(Q)$.

To choose the number $S$ we first fix an integer $k\geq r$. Let us estimate the probability that there exists a cube $R\in\mathcal{D}_2$ of size $\ell(R)=2^k\ell(Q)$ such that $\dist(Q, \partial R)\le \ell(Q)^{\al}\ell(R)^{1-\al}$.  Geometric reasoning show it is equal to the ratio of the area of the narrow strip of size $2^{(1-\alpha)k}\ell(Q)$ around the boundary of the cube $R$ to the whole area of $2^k Q$.  The area of the strip around the boundary is at most $C|Q| 2^{dk-\alpha k}$ (here $C=C(d)$ is an absolute constant depending only on the dimension $d$).  We conclude that this ratio is less than
$C\,2^{-k\al}$.  Therefore, the probability that the cube $Q$ is bad  does not exceed
$$
C\,\sum_{k=r}  ^\infty 2^{-k\al}=\frac{C\,\cdot 2^{-r\al}}{1-2^{-\al}}\,.
$$
Since the integer $r$ was chosen so that $2^{-r}\leq \delta^S<2^{-r+1}$ we will then simply choose the minimal $S=S(\alpha)$ such that $\frac{C\, \delta^{S\al}}{1-2^{-\al}}\leq \delta^2$ (of course, $S=3/\al$ is enough for all small $\delta$'s).
\end{proof}

The use of Theorem \ref{badprob} gives us
 $S= S(\alpha)$ in such a way that for any fixed $Q\in \mathcal{D}_1$,
 \begin{equation}
 \label{probagain}
 \mathbb{P}_{\om_2}\{Q \,\text{is bad}\} \leq \delta^2\,.
 \end{equation}
We are now ready to prove

\begin{theorem}
\label{badprobfagain}
Consider the decomposition of $f$ from Lemma \ref{Riesz} and the resulting $f_{bad}$ arising from every $\om_2\in\Omega$. One can choose $S= S(\alpha)$ in such a way that 
\begin{equation}
\label{probf}
\mathbb{E}(\|f_{bad}\|_{L^2(\R^d;\mu)}) \leq \delta\|f\|_{L^2(\R^d;\mu)}\,.
\end{equation}
\end{theorem}
The proof depends only on the property \eqref{probagain} and not on a particular definition of what it means to be a bad or good cube.
\begin{proof}
By Lemma \ref{Riesz} (its left inequality),
$$
\mathbb{E}(\|f_{bad}\|_{L^2(\R^d;\mu)}) \leq \mathbb{E}\Big(\sum_{Q\in\Dk_1^{tr}\cap\mathcal{B}_1}\|\Delta_Q f\|^2_{L^2(\R^d;\mu)}\Big)^{1/2}\,.
$$
Then
$$
\mathbb{E}(\|f_{bad}\|_{L^2(\R^d;\mu)}) \leq \Big(\mathbb{E}\sum_{Q\in\Dk_1^{tr}\cap\mathcal{B}_1}\|\Delta_Q f\|^2_{L^2(\R^d;\mu)}\Big)^{1/2}\,.
$$
Let $Q$ be a fixed cube in $\mathcal{D}_1$; then, using \eqref{probagain}, we conclude: 
$$
\mathbb{E}_{\om_2}\|\Delta_Q f\|^2_{L^2(\R^d;\mu)}=\mathbb{P}_{\om_2}\{Q\,\text{is bad}\} \|\Delta_Q f\|^2_{L^2(\R^d;\mu)} \leq \delta^2 \|\Delta_Q f\|^2_{L^2(\R^d;\mu)}\,.
$$
 Therefore, we can continue as follows:
 $$
 \mathbb{E}(\|f_{bad}\|_{L^2(\R^d;\mu)}) \leq \delta\Big(\sum_{Q\in\Dk_1^{tr}\cap\mathcal{B}_1}\|\Delta_Q f\|^2_{L^2(\R^d;\mu)}\Big)^{1/2}\leq \delta\|f\|_{L^2(\R^d;\mu)}\,.
 $$
 The last inequality uses Lemma \ref{Riesz} again (its right inequality).

\end{proof}

\subsection{Reduction to Estimates on Good Functions}
\label{redtogood}

We consider $\om_1,\om_2\in \Omega$, two dyadic lattices $\Dk_1$ and $\Dk_2$ corresponding to the shifts of the standard lattice by $\om_1$ and $\om_2$.  Now take two functions $f$ and $g\in L^2(\R^d;\mu)$ decomposed according to Lemma \ref{Riesz} 
$$
f=\La f + \sum_{Q\in \Dk_1^{tr}} \Delta_Q f\,,\,\,g=\La g + \sum_{R\in \Dk_2^{tr}} \Delta_R g\,.
$$
which we further decompose into $f=f_{good} + f_{bad}$, and similarly $g=g_{good}+ g_{bad}$. Let $T$ stand for any operator with a bounded kernel. In the future it can be, for example, the operator with the kernel ($\eta>0$) (and then appropriately symmetrized)
$$
k(z,w) = \frac{(\bar z -w)^2}{|z-\bar w|^4 + \eta^2}\,,\,\, z, w\in \mathbb{C}_+\,.
$$
Then 
$$
(Tf,g) = (Tf_{good}, g_{good}) + R_{(\om_1,\om_2)}(f,g)\,,\,\,\text{where}\,\, R_{(\om_1,\om_2)}(f,g)= (Tf_{bad}, g) + (Tf_{good}, g_{bad})\,.
$$

\begin{theorem}
\label{R}
Let $T$ be any operator with bounded kernel. Then
$$
\mathbb{E} |R_{(\om_1,\om_2)}(f, g)| \le 2\,\delta\|T\|_{L^2(\R^d;\mu)\to L^2(\R^d;\mu)} \|f\|_{L^2(\R^d;\mu)} \|g\|_{L^2(\R^d;\mu)}\,.
$$ 
\end{theorem}

\medskip

\begin{remark}
Notice that the estimate depends on the norm of $T$ and not on the bound on its kernel!
\end{remark}

\medskip

\begin{proof}
Observe that taking the good or bad part of a function are projections in $L^2(\R^d;\mu)$, and so they do not increase the norm.  Since we have that the operator $T$ is bounded, then
$$
|R_{(\om_1,\om_2)}(f,g)|\leq \|T\|_{L^2(\R^d;\mu)\to L^2(\R^d;\mu)}\left(\|g\|_{L^2(\R^d;\mu)}\|f_{bad}\|_{L^2(\R^d;\mu)} + \|f\|_{L^2(\R^d;\mu)}\|g_{bad}\|_{L^2(\R^d;\mu)}\right)
$$ 
Therefore, upon taking expectations we find
$$
\mathbb{E}|R_{(\om_1,\om_2)}(f,g)| \le \|T\|_{L^2(\R^d;\mu)\to L^2(\R^d;\mu)}\left(\|g\|_{L^2(\R^d;\mu)}\mathbb{E}(\|f_{bad}\|_{L^2(\R^d;\mu)}) + \|f\|_{L^2(\R^d;\mu)}\mathbb{E}(\|g_{bad}\|_{L^2(\R^d;\mu)})\right)\,.
$$
Using Theorem \ref{badprobfagain} we finish the proof.

\end{proof}

Theorem \ref{R} implies that we only need to obtain the following estimate:
\begin{equation}
\label{good}
|(Tf_{good}, g_{good})| \le C(\tau,m,d, A) \|f\|_{L^2(\R^d;\mu)}\|g\|_{L^2(\R^d;\mu)}\,.
\end{equation}
In fact, considering any operator $T$ with bounded kernel we conclude
$$
(Tf,g) = \mathbb{E}(Tf,g) = \mathbb{E}(Tf_{good}, g_{good}) + \mathbb{E} R_{(\om_1,\om_2)}(f,g)\,.
$$
Using Theorem \ref{R} and \eqref{good} we have
$$
|(Tf,g)| \le C\,\|f\|_{L^2(\R^d;\mu)}\|g\|_{L^2(\R^d;\mu)} + 2\delta\|T\|_{L^2(\R^d;\mu)\to L^2(\R^d;\mu)}\|f\|_{L^2(\R^d;\mu)}\|g\|_{L^2(\R^d;\mu)}\,.
$$
From here, taking the supremum over $f$ and $g$ in the unit ball of $L^2(\R^d;\mu)$, and choosing $\delta=\frac14$ we get
$$
\|T\|_{L^2(\R^d;\mu)\to L^2(\R^d;\mu)}\le 2C\,.
$$

\subsection{Splitting $(Tf_{good}, g_{good})$ into Three Sums}
\label{spli}

Our next reduction is to show that we can only deal with the case when $f$ and $g$ are good functions and with zero average.  To accomplish this, we will remove the term from $f_{good}$ corresponding to the term $\La$. 

We fix $\om_1,\om_2\in \Omega$, and the two corresponding dyadic lattices $\Dk_1$ and $\Dk_2$, and recall that $F=\supp\mu$ is deep inside a unit cube $Q$ and shifted unit cubes $Q^0\in \Dk_1, R^0\in \Dk_2$.

If $f\in L^2(\R^d;\mu)$, we have by the ``T1 Condition'', Condition (\ref{T1}), and an application of Cauchy--Schwarz that
$$
\|T\La f\|_{L^2(\R^d;\mu)} = \langle f\rangle_{Q^0} \|T\chi_{Q^0}\|_{L^2(\R^d;\mu)} \le A\frac{\|f\|_{L^2(\R^d;\mu)} \sqrt{\mu(Q^0))}}
{\mu(Q^0)} \sqrt{\mu(Q^0)}=A\|f\|_{L^2(\R^d;\mu)}\,.
$$

So we can replace $f$ by $f-\La f$. These computations of can course be repeated for $g$ and so from now on we may assume that
$$
\int_{\mathbb{R}^d} f(x)\,d\mu(x)=0\textnormal{ and }\int_{\mathbb{R}^d} g(x)\,d\mu(x)=0\,.
$$

We skip mentioning below that $Q\in \Dk_1^{tr}$ and $R\in \Dk_2^{tr}$ since this will always be the case.  Computing the inner product between $Tf$ and $g$ we see,
\begin{align*}
(Tf,g)= \sum_{Q \in\mathcal{G}_1, R\in \mathcal{G}_2 , \ell(Q) \le \ell(R)} (\Delta_Q f, \Delta_R g) +
\sum_{Q \in\mathcal{G}_1, R\in \mathcal{G}_2, \ell(Q) > \ell(R)} (\Delta_Q f, \Delta_R g) \,.
\end{align*}
The question of convergence of the infinite sum can be avoided here, as we can think that the decompositions of $f$ and $g$ are only into a finite sum of terms (notice that the decomposition of a good function is also good).
This removes the question of convergence and allows us to rearrange and group the terms in the sum in the way that is most convenient for our interests.

Clearly, we need to estimate only the first sum above, the second will follow by symmetry. For the sake of notational simplicity we will skip mentioning that the cubes $Q$ and $R$ are good and we will skip mentioning $\ell(Q)\le \ell(R)$. So, from now on, 
$$
\sum_{Q,R: \text{other conditions}} \textnormal{ simply means } \sum_{Q,R: \ell(Q)\le \ell(R),\, Q \in\mathcal{G}_1, \,R\in\mathcal{G}_2 , \,\text{other conditions}}\,.
$$

\medskip

\begin{remark}
It is convenient sometimes to think that the summation 
$$
\sum_{Q,R: \,\text{other conditions}}
$$
 goes over good $Q$ and \textit{all} $R$. Formally, it does not matter, since $f$ and $g$ are good functions, it is merely the matter of adding or omitting several zeros. For the symmetric sum over $Q$ and $R$ with the property $\ell(Q)>\ell(R)$ the roles of $Q$ and $R$ in this remark must be interchanged.
\end{remark}

\medskip

The definition of $\delta$-badness involved a large integer $r$, see \eqref{r}. Using this notation, we write our sum over $\ell(Q)\le \ell(R)$ as follows (we suppress $(\Delta_Qf, \Delta_R g)$ in the display below for ease of presentation):

\begin{eqnarray*}
\sum_{Q,R} (\Delta_Qf, \Delta_R g) & = & \sum_{Q,R: \ell(Q)\ge 2^{-r} \ell(R)} + \sum_{Q,R: \ell(Q)< 2^{-r} \ell(R)}\\
& = & \sum_{Q,R: \ell(Q)\ge 2^{-r} \ell(R),\,\dist(Q,R) \le \ell(R)}+ \sum_{Q,R: \ell(Q)< 2^{-r} \ell(R),\,Q\cap R\neq\emptyset}\\
& + & \bigg[\sum_{Q,R: \ell(Q)\ge 2^{-r} \ell(R),\,\dist(Q,R)>\ell(R)}+\sum_{Q,R: \ell(Q)< 2^{-r} \ell(R),\,Q\cap R=\emptyset}\bigg].
\end{eqnarray*}
We then define
\begin{eqnarray}
\label{sigma1}\sigma_1 & := & \sum_{Q,R: \ell(Q)\ge 2^{-r} \ell(R),\,\dist(Q,R) \le \ell(R)}(\Delta_Qf, \Delta_R g)\\
\label{sigma2}\sigma_2 & := & \bigg[\sum_{Q,R: \ell(Q)\ge 2^{-r} \ell(R),\,\dist(Q,R)>\ell(R)}+\sum_{Q,R: \ell(Q)< 2^{-r} \ell(R),\,Q\cap R=\emptyset}\bigg](\Delta_Qf, \Delta_R g)\\
\label{sigma3}\sigma_3 & := & \sum_{Q,R: \ell(Q)< 2^{-r} \ell(R),\,Q\cap R\neq\emptyset}(\Delta_Qf, \Delta_R g)
\end{eqnarray}

\subsection{Potential ways to estimate $\int_{\mathbb{R}^d}\int_{\mathbb{R}^d} k(x,y) f(x)g(y)\,d\mu(x)\,d\mu(y)$}
\label{threetypes}
Recall that the kernel $k(x,y)$ of $T$ satisfies the estimate
$$
|k(x,y)|\le \frac{1}{\max (d(x)^m, d(y)^m)}\,,\,\,\,d(x) =\dist (x,\R^d\setminus H)\,,
$$ $H$ being an open set in $\R^d$,
 and
$$
|k(x,y)|\le\frac{C_{CZ}}{|x-y|^m}\,.
$$
This above inequality implies that
$$
|k(x,y)-k(x',y)|\le C_{CZ}\frac{|x-x'|^\tau}{|x-y|^{m+\tau}}
$$
provided that $|x-x'|\le\tfrac12|x-y|$,with some (fixed) $0<\tau\le 1$ and $0<C_{CZ}< \infty$.
We will sometimes write 
$$
\int_{\mathbb{R}^d}\int_{\mathbb{R}^d} k(x,y) f(x)g(y)\,d\mu(x)\,d\mu(y)=\int_{\mathbb{R}^d}\int_{\mathbb{R}^d} [k(x,y)-k(x_0,y)]f(x)g(y)\,d\mu(x)\,d\mu(y)
$$
using the fact that our $f$ and $g$ will be $\Delta_Q f$ and $\Delta_R g$ and so their
integrals are zero. Temporarily call $K(x,y)$ either $k(x,y)$ or $k(x,y)-k(x_0,y)$.

After that we have three logical ways to estimate 
$$
\int_{\mathbb{R}^d}\int_{\mathbb{R}^d} K(x,y) f(x)g(y)\,d\mu(x)\,d\mu(y):
$$

\begin{itemize}
\item[(1)] Estimate $|K|$ in $L^{\infty}$, and $f, g$ in the $L^1$ the norm;
\item[(2)] Estimate $|K|$ in $L^{\infty}\times L^1$ norm, and $f$ in $L^1$ norm, $g$ in $L^{\infty}$ norm (or maybe, do this symmetrically); 
\item[(3)] Estimate $|K|$ in $L^1$ norm, and $f, g$ in the $L^{\infty}$ norm.
\end{itemize}

The third method is widely used for Calder\'on--Zygmund estimates on homogeneous spaces (say with respect to Lebesgue measure), but this estimate is very dangerous to use for non-homogeneous measures. To see why, consider the following reason.  After $f$ and $g$ are estimated in the $L^{\infty}$ norm, one needs to continue these estimates to have $L^2$ norms. There is nothing strange about this since typically $f$ and $g$ are almost proportional to characteristic functions. But, for $f$ living on $Q$ such that $f=c_Q\chi_Q$ ($c_Q$ is a constant),
$$
\|f\|_{L^{\infty}(\R^d;\mu)} \lesssim \frac1{\sqrt{\mu(Q)}}\|f\|_{L^2(\R^d;\mu)}\,.
$$
The same reasoning applies for $g$ on $R$. Then
$$
\Big|\int_{\mathbb{R}^d}\int_{\mathbb{R}^d} K(x,y) f(x)g(y)\,d\mu(x)\,d\mu(y)\Big|\lesssim \frac1{\sqrt{\mu(Q)}\sqrt{\mu(R)}}\|f\|_{L^2(\R^d;\mu)}\|g\|_{L^2(\R^d;\mu)}\,.
$$
However, a non-homogeneous measure has no estimate from below and having two uncontrollable almost zeroes in the denominator is a very bad idea.  We will never use the estimate of type (3). 

On the other hand, estimates of type (2) are much less dangerous (although care is still required).  Because, in this case one applies 
 $$
 \|f\|_{L^1(\R^d;\mu)} \lesssim \sqrt{\mu(Q)}\|f\|_{L^2(\R^d;\mu)},\quad \|g\|_{L^{\infty}(\R^d;\mu)} \lesssim\frac1{\sqrt{\mu(R)}}\|g\|_{L^2(\R^d;\mu)}\,,
$$
 and gets $$\Big|\int_{\mathbb{R}^d}\int_{\mathbb{R}^d} K(x,y) f(x)g(y)\,d\mu(x)\,d\mu(y)\Big| \lesssim\frac{\sqrt{\mu(Q)}}{\sqrt{\mu(R)}}\|f\|_{L^2(\R^d;\mu)}\|g\|_{L^2(\R^d;\mu)}\,.
 $$
 If we choose to use estimates of type (2) only for pairs $Q,R$ such that $Q\subset R$ we are in good shape. This is what we will be doing when estimating $\sigma_3$.
 
Finally, estimates of type (1) yield the following
$$
\Big|\int_{\mathbb{R}^d}\int_{\mathbb{R}^d} K(x,y) f(x)g(y)\,d\mu(x)\,d\mu(y)\Big|\lesssim \sqrt{\mu(Q)}\sqrt{\mu(R)}\|f\|_{L^2(\R^d;\mu)}\|g\|_{L^2(\R^d;\mu)}\,.
$$
These present no problems from the measure $\mu$.  But, estimates of this type will only be good when we can exploit decay from the cubes $R$ and $Q$ being far apart.  This will be the strategy for handling term $\sigma_2$.
 
\bigskip

\noindent{\bf Plan of the Proof:} The first sum is the ``diagonal" part of the operator, $\sigma_1$. The second sum, $\sigma_2$ is the ``long range interaction".  The final sum, $\sigma_3$, is the ``short range interaction".  The diagonal part will be estimated using our $T1$ assumption of Theorem \ref{T1}, for the long range interaction we will use the first type of estimates described above, for the short range interaction we will use estimates of types (1) and (2) above. But, all this will be done carefully!

\section{The Long Range Interaction: Controlling Term $\sigma_2$}
\label{lr}

We first prove a lemma that demonstrates that for functions with supports that are far apart, we have some good control on the bilinear form induced by our \cz\, operator $T$.  For two dyadic cubes $Q$ and $R$, we set
$$
D(Q,R):=\ell(Q)+\ell(R)+\dist(Q,R).
$$ 

\begin{lemma} 
  \label{FarInteractionLemma}
  Suppose that $Q$ and $R$ are two cubes in $\R^d$, such that
  $\ell(Q) \leq \ell(R)$. Let $\vf\ci Q,\psi\ci R\in L^2(\R^d;\mu)$. Assume that $\vf\ci Q$ vanishes outside $Q$, $\psi\ci R$
  vanishes outside $R$; $\int_{\R^d}\vf\ci Qd\mu=0$ and, last, $\dist(Q,\supp\psi\ci R)\ge \ell(Q)^{\al}\ell(R)^{1-\al}$. Then
  $$|( \vf\ci Q,T\psi\ci R )|\le A\,C\,\frac{\ell(Q)^{\frac{\tau}{2}}\ell(R)^{\frac{\tau}{2}}}{D(Q,R)^{m+\tau}}\sqrt{\mu(Q)}\sqrt{\mu(R)}\|\vf\ci Q\|\ci{L^2(\R^d;\mu)}\|\psi\ci R\|\ci{L^2(\R^d;\mu)}.
  $$
  \end{lemma} 
  
\medskip  
  
\begin{remark}
  Note that we require only that the support of the function $\psi_R$ lies far from $Q$; the cubes $Q$ and $R$ themselves may intersect! This situation will really occur when estimating $\sigma_2$.
\end{remark}  

\medskip

  \begin{proof}

Let $x\ci Q$ be the center of the cube $Q$. Note that for all $x\in Q$,
$y\in \supp\psi\ci R$, we have
$$
|x\ci Q-y|\ge \frac{\ell(Q)}{2}+\dist(Q,\supp\psi\ci R)\ge
\frac{\ell(Q)}{2} + 2^{r(1-\al)}\ell(Q)\geq C(r,\alpha)\ell(Q)\ge 2|x-x\ci Q|.
$$
Therefore,
\begin{align*}
|( \vf\ci Q, T \psi\ci R )|&=
\Bigl|\int_{\mathbb{R}^d}\int_{\mathbb{R}^d} k(x,y)\vf\ci Q(x)\psi\ci R(y)\, d\mu(x)\, d\mu(y)\Bigr|\\
&=
\Bigl|\int_{\mathbb{R}^d}\int_{\mathbb{R}^d} [k(x,y)-k(x\ci Q,y)]
\vf\ci Q(x)\psi\ci R(y)\, d\mu(x)\, d\mu(y)\Bigr|\\
&\leq
C_{CZ}\frac{\ell(Q)^\tau}{\dist(Q,\supp\psi\ci R)^{m+\tau}}
\|\vf\ci Q\|\ci{L^1(\R^d;\mu)}\|\psi\ci R\|\ci{L^1(\R^d;\mu)}.
\end{align*}
There are two possible cases.

{\bf Case 1: }The cubes $Q$ and $R$ satisfy $\dist(Q,\supp\psi\ci R)\ge \ell(R)$.  If this holds, then we have 
$$
D(Q,R):=\ell(Q)+\ell(R)+\dist(Q,R)\le
3\dist(Q,\supp\psi\ci R)
$$
and therefore
$$
\frac{\ell(Q)^\tau}{\dist(Q,\supp\psi\ci R)^{m+\tau}}\le C\,
\frac{\ell(Q)^\tau}{D(Q,R)^{m+\tau}}\le
C\,\frac{\ell(Q)^{\frac{\tau}{2}}\ell(R)^{\frac{\tau}{2}}}{D(Q,R)^{m+\tau}}.
$$

{\bf Case 2:}  The cubes $Q$ and $R$ satisfy $\ell(Q)^{\al}\ell(R)^{1-\al}\le \dist(Q,\supp\psi\ci R)\le \ell(R)$.

Then $D(Q,R)\le 3\ell(R)$ and we get
$$
\frac{\ell(Q)^\tau}{\dist(Q,\supp\psi\ci R)^{m+\tau}}\le
\frac{\ell(Q)^\tau}{[\ell(Q)^{\al}\ell(R)^{1-\al}]^{m+\tau}}=
\frac{\ell(Q)^{\frac{\tau}{2}}\ell(R)^{\frac{\tau}{2}}}{\ell(R)^{m+\tau}}\le
C\,\frac{\ell(Q)^{\frac{\tau}{2}}\ell(R)^{\frac{\tau}{2}}}{D(Q,R)^{m+\tau}}.
$$
The computation of the equality at this point, uses the particular choice of $\alpha=\frac{\tau}{2(\tau+m)}$.  Now, to finish the proof of the lemma, it remains only to note that
$$
\|\vf\ci Q\|\ci{L^1(\R^d;\mu)}\le
\sqrt{\mu(Q)}\|\vf\ci Q\|\ci{L^2(\R^d;\mu)}
\text{\qquad and\qquad }
\|\psi\ci R\|\ci{L^1(\R^d;\mu)}\le
\sqrt{\mu(R)}\|\psi\ci R\|\ci{L^2(\R^d;\mu)}.
$$

\end{proof}

Applying this lemma to $\vf\ci Q=\Delta\ci Q f$ and $\psi\ci R=\Delta\ci R g$, we
obtain
\begin{equation}
\label{twostar} 
|\sigma_2|\le
C\,
\sum_{Q,R}
\frac{\ell(Q)^{\frac{\tau}{2}}\ell(R)^{\frac{\tau}{2}}}{D(Q,R)^{m+\tau}}
\sqrt{\mu(Q)}\sqrt{\mu(R)}\|\Delta\ci Q f\|\ci{L^2(\R^d;\mu)}
\|\Delta\ci R g\|\ci{L^2(\R^d;\mu)}\,.
\end{equation}

\smallskip

If we let $T\ci{Q,R}$ denote the matrix defined by
$$
T\ci{Q,R}:=
\frac{\ell(Q)^{\frac{\tau}{2}}\ell(R)^{\frac{\tau}{2}}}{D(Q,R)^{m+\tau}}
\sqrt{\mu(Q)}\sqrt{\mu(R)}\qquad (Q\in\mathcal{D}_1^{tr}  ,\,
R\in\mathcal{D}_2^{tr}  ,\,\ell(Q)\leq\ell(R)\,)
$$
then we will show that the corresponding operator generates a bounded operator in $l^2$.  This then will allow us to control the term $|\sigma_2|$.  Since if the operator $T\ci{Q,R}$ is bounded on $l^2$ we have
\begin{eqnarray*}
|\sigma_2| & \leq & \| T\ci{Q,R}\|_{l^2\to l^2}\left[\sum_{Q}\|\Delta\ci{Q} f\|_{L^2(\R^d;\mu)}^2\right]^{\frac 12}\left[\sum_{R}\|\Delta\ci{R} g\|_{L^2(\R^d;\mu)}^2\right]^{\frac 12}\\
 & \leq &  \| T\ci{Q,R}\|_{l^2\to l^2}\|f\|_{L^2(\R^d;\mu)}\|g\|_{L^2(\R^d;\mu)}.
\end{eqnarray*}
The last inequality follows by Lemma \ref{Riesz}.

\begin{lemma}
\label{TQR}
For any two
families  $\{a\ci Q\}\ci{Q\in\mathcal{D}_1^{tr}  }$ and $\{b\ci R\}\ci{R\in\mathcal{D}_2^{tr}  }$
of nonnegative numbers, one has
$$
\sum_{Q,R}T\ci{Q,R}a\ci Q b\ci R\le
A\,C\,\Bigl[\sum_{Q}a\ci Q^2\Bigr]^{\frac12}\Bigl[\sum_{R}b\ci R^2\Bigr]^{\frac12}.
$$
\end{lemma}

\medskip

\begin{remark}
Note that $T\ci{Q,R}$ are defined for all $Q$ and $R$ with $\ell(Q)\le \ell(R)$
and that the conditions $\dist(Q,R)\ge \ell(Q)^{\al}\ell(R)^{1-\al}$ (or even the condition
$Q\cap R=\emptyset$) no longer appears in the summation!
\end{remark}

\medskip

\begin{proof}
Let us ``slice'' the matrix $T\ci{Q,R}$ according to the ratio $\frac{\ell(Q)}{\ell(R)}$. Namely, let
$$
T^{(k)}_{Q,R}:=\left\{\aligned T\ci{Q,R}&\qquad\text{if } \ell(Q)=2^{-k}\ell(R)\,;\\0&\qquad\text{otherwise}\,,\endaligned\right.
$$
($k=0,1,2,\dots$).
To prove the lemma, it is enough to show that for every $k\ge 0$,
$$
\sum_{Q,R}T^{(k)}\ci{Q,R}a\ci Q b\ci R\le C\,2^{-\frac{\tau}{2}k}\Bigl[\sum_{Q}a\ci Q^2\Bigr]^{\frac12}\Bigl[\sum_{R}b\ci R^2\Bigr]^{\frac12}.
$$
The matrix $\{T^{(k)}_{Q,R}\}$ has a ``block'' structure: the terms $b\ci R$ corresponding to the cubes $R\in\mathcal{D}_2^{tr}$ for which $\ell(R)=2^{j}$ can interact only with the terms $a\ci Q$ corresponding to the cubes $Q\in\mathcal{D}_1^{tr} $, for which $\ell(Q)=2^{j-k}$. Thus, to get the desired inequality, it is enough to estimate each block separately, i.e., to demonstrate that
$$
\sum_{Q,R\,:\,\ell(Q)=2^{j-k},\ell(R)=2^j}T^{(k)}\ci{Q,R}a\ci Q b\ci R\leq C\,\Bigl[\sum_{Q\,:\,\ell(Q)=2^{j-k}}a\ci Q^2\Bigr]^{\frac12}
\Bigl[\sum_{R\,:\,\ell(R)=2^j}b\ci R^2\Bigr]^{\frac12}.
$$
Let us introduce the functions
$$
F(x):=\sum_{Q\,:\,\ell(Q)=2^{j-k}}\frac{a\ci Q}{ \sqrt{\mu(Q)} }
\chi\ci Q(x)
\qquad
\text{and}
\qquad
G(x):=\sum_{R\,:\,\ell(R)=2^{j}}\frac{b\ci R}{ \sqrt{\mu(R)} }
\chi\ci R(x).
$$
Note that the cubes of a given size in one dyadic lattice do not intersect,
and therefore at each point $x\in\R^d$, at most one term in the sum can be
non-zero. Also observe that
$$
\|F\|\ci{L^2(\R^d;\mu)}=
\Bigl[\sum_{Q\,:\,\ell(Q)=2^{j-k}}a\ci Q^2\Bigr]^{\frac12}
\qquad
\text{and}
\qquad
\|G\|\ci{L^2(\R^d;\mu)}=
\Bigl[\sum_{R\,:\,\ell(R)=2^{j}}b\ci R^2\Bigr]^{\frac12}.
$$
Then the estimate we need can be rewritten as
$$
\int_{\mathbb{R}^d}\int_{\mathbb{R}^d} K_{j,k}(x,y)F(x)G(y)\, d\mu(x)\,d\mu(y) \le
C\,
\|F\|\ci{L^2(\R^d;\mu)}\|G\|\ci{L^2(\R^d;\mu)},
$$
where
$$
K_{j,k}(x,y)=\sum_{Q,R\,:\,\ell(Q)=2^{j-k}, \ell(R)=2^j}
\frac{\ell(Q)^{\frac{\tau}{2}}\ell(R)^{\frac{\tau}{2}}}{D(Q,R)^{m+\tau}}\chi\ci Q(x)\chi\ci
R(y).
$$
Again, for every pair of points $x,y\in \R^d$, only one term in the sum can be
nonzero.
Since $|x-y|+\ell(R)\le 3D(Q,R)$ for any $x\in Q$, $y\in R$, we obtain
\begin{align*}
K_{j,k}(x,y)&=
C\,2^{-\frac{\tau}{2}k}\frac{ \ell(R)^{\tau} }{ D(Q,R)^{m+\tau}}\\
&\leq
C\,2^{-\frac{\tau}{2}k}
\frac{2^{j\tau}}{[2^j+|x-y|]^{m+\tau}}
=:C\,2^{-\frac{\tau}{2}k} k_j(x,y).
\end{align*}
So, it is enough to check that
$$
\int_{\mathbb{R}^d}\int_{\mathbb{R}^d} k_{j}(x,y)F(x)G(y)\, d\mu(x)\,d\mu(y)
\le
C\,
\|F\|\ci{L^2(\R^d;\mu)}\|G\|\ci{L^2(\R^d;\mu)}.
$$
Recall that  we called the balls ``non-Ahlfors balls" if
$$
\mu(B(x,r)) >r^m\,.
$$
According to the Schur test, it would suffice to prove that
for every $y\in \R^d$, one has the estimate $\int_{\R^d}k_j(x,y)\,d\mu(x)\le C$ and vice versa (i.e., for every $x\in \R^d$, one has $\int_{\R^d}k_j(x,y)\,d\mu(y)\le C$). Then the norm of the integral operator with kernel $k_j$ in $L^2(\R^d;\mu)$ would be bounded by the same constant $C$, and the proof of Lemma \ref{TQR} would be over.
If we assumed a priori that the supremum of radii of all  non-Ahlfors balls  centered at $y\in R$ with $\ell(R) =2^j,$ were less than $2^{j+1}$, then the needed estimate would be immediate. In fact, we could write
\begin{align*}
\int_{\R^d}k_j(x,y)\,d\mu(x)&=\int_{B(y,2^{j+1})}k_j(x,y)\,d\mu(x)+\int_{\R^d\setminus B(y,2^{j+1})}k_j(x,y)\,d\mu(x)\\&\lesssim 2^{-jm}\mu(B(y,2^{j+1}))+\int_{\R^d\setminus B(y,2^{j+1})}\frac{2^{j\tau}}{|x-y|^{m+\tau}}\,d\mu(x)\\&
\lesssim \Bigl(1 + \int_{2^j}^{+\infty}\frac{2^{j\tau}t^{m-1}}{t^{m+\tau}}dt\Bigr)=C\,.
\end{align*}

The problem is that we cannot guarantee that the supremum of radii of all  non-Ahlfors balls  centered at $y$
be less than $2^{j+1}$ for every $y\in\R^d$. Our measure may not have this uniform property.

So, generally speaking, we are unable to show that the
integral operator with kernel $k_j(x,y)$ acts in $L^2(\R^d;\mu)$. But we {\it do not need} that much! We only need to check that the corresponding bilinear form is bounded on two {\it given} functions $F$ and $G$. So, we are not interested in the points $y\in\R^d$ for which $G(y)=0$ (or in the points $x\in\R^d$, for which $F(x)=0$). But, by  definition, $G$ can be non-zero only on transit cubes in $\mathcal{D}_2$. 
Here we use our convention that we  omit in all sums the fact that $Q,R$ are transit cubes. But they are!

Now let us notice that if (and this is the case for all $R$ in the sum we estimate in our lemma) $R\in\mathcal{D}_2^{tr}  $, then the supremum of radii of all non-Ahlfors balls  centered at $y\in R$ is bounded by $c(d)\ell(R)$ for every ${y\in R}$. Indeed, this is just Lemma \ref{obv2}. The same reasoning shows that if $Q\in\mathcal{D}_1^{tr}  $, then the supremum of radii of all non-Ahlfors balls  centered at $x\in Q$ is bounded by $2^{j-k+1}\lesssim 2^{j+1}$ whenever $F(x)\ne 0$, and we are done with Lemma \ref{TQR}.
\end{proof}

Now, we hope, the reader will agree that the decision to declare the cubes contained in $H$ terminal was a good one.  As a result, the fact that the measure $\mu$ is not Ahlfors did not put us in any real trouble -- we barely had a chance to notice this fact at all. But, it still remains to explain why we were so eager to have the extra condition
$$
|k(x,y)|\le \frac{1}{\max (d(x)^m, d(y)^m)}\,,\,\, d(x):= \dist (x, \R^d\setminus H)
$$
 on our Calder\'on--Zygmund kernel.
  The answer is found in the next two sections.

  \section{Short Range Interaction and Nonhomogeneous Paraproducts: Controlling Term $\sigma_3$.}
  \label{shortrange}
  
  Recall that the sum $\sigma_3$ is taken over the pairs $Q,R$, for which $\ell(Q)<2^{-r}  \ell(R)$ and $Q\cap R\ne\emptyset$. We would like to improve this condition and demand that $Q$ lie ``deep inside'' one of the $2^d$ subcubes $R_j$ ($j=1,\ldots, 2^d$).
  Recall also that we defined the {\it skeleton} $sk\, R$ of the cube $R$ by
  $$
  sk\, R:=\bigcup_{j=1}^{2^d}\partial R_j.
  $$
  We have declared a cube $Q\in\mathcal{D}_1$ bad if there exists a cube $R\in\mathcal{D}_2$
  such that $\ell(R)>2^r \ell(Q)$ and $\dist(Q, sk\, R)\le\ell(Q)^{\al}\ell(R)^{1-\al}$. Now, for every good cube $Q\in\mathcal{D}_1$, the conditions $\ell(Q)<2^{-r}  \ell(R)$ and $Q\cap R\ne \emptyset$
together imply
 that $Q$ lies inside one of the $2^d$ children $R_j$ of $R$. We will denote
this subcube by $R\ci Q$. The sum $\sigma_3$ can now be split into
$$
\sigma_3^{term}:=\sum_{ Q,R\,:\,Q\subset R,\, \ell(Q)<2^{-r}  \ell(R),
                    \\ R\ci Q\text{ is terminal}}
(\Delta\ci Q f, T\Delta\ci R g)
$$
and
$$
\sigma_3^{tr}  :=\sum_{ Q,R\,:\,Q\subset R,\, \ell(Q)<2^{-r}  \ell(R),
                    \\  R\ci Q\text{ is transit} }
(\Delta\ci Q f, T\Delta\ci R g).
$$

\smallskip

\subsection{ Estimation of $\sigma_3^{term}$.}
\label{sigmaterm}

First of all, write (recall that $R_j$ denotes a child of $R$):
$$
\sigma_3^{term}=\sum_{j=1}^{2^{n}}\,\,
\sum_{ Q,R\,:\,\ell(Q)<2^{-r}  \ell(R),
\\
Q\subset R_j\in\mathcal{D}_2^{term}}
(\Delta\ci Q f, T\Delta\ci R g).
$$
Clearly, it is enough to estimate the inner sum for every fixed $j=1,\dots ,2^d$. Let us
do this for $j=1$. We have
$$
\sum_{ Q,R\,:\,\ell(Q)<2^{-r}  \ell(R), \\ Q\subset
R_1\in\mathcal{D}_2^{term}}
(\Delta\ci Q f, T\Delta\ci R g)=
\sum_{R:R_1\in\mathcal{D}_2^{term}}\,\, \sum_{ Q:\,\ell(Q)<2^{-r}  \ell(R), 
Q\subset R_1}
(\Delta\ci Q f, T\Delta\ci R g).
$$
Recall that the kernel $k$ of our operator $T$ satisfies
the estimate of Lemma \ref{obv1}
\begin{equation}
\label{konwholeR1}
|k(x,y)|  \le\frac{1}{\ell(R)^m}\qquad\text{for all }
x\in R_1, y\in\R^d.
\end{equation}
Hence,
\begin{equation}
\label{TonwholeR1}
|T\Delta\ci R g(x)|\le\frac{\|\Delta\ci R g\|\ci{L^1(\R^d;\mu)}}{\ell(R)^m}
\qquad\text{ for all }x\in R_1,
\end{equation}
and therefore
\begin{align*}
\|\chi\ci{R_1}\cdot T\Delta\ci R g \|\ci{L^2(\R^d;\mu)}&\le
\|\Delta\ci R g\|\ci{L^1(\R^d;\mu)}\frac{\sqrt{\mu(R_1)}}{\ell(R)^m}\\
&\le
\frac{{\mu(R)}}{\ell(R)^m}\|\Delta\ci R g\|\ci{L^2(\R^d;\mu)}\le
C\|\Delta\ci R g \|\ci{L^2(\R^d;\mu)}.
\end{align*}
This follows because $\|\Delta\ci R g \|\ci{L^1(\R^d;\mu)}
\le \sqrt{\mu(R)} \|\Delta\ci R g \|\ci{L^2(\R^d;\mu)}$ and $\mu(R_1)\le \mu(R)$ hold trivially.  Additionally, we used the observation that, by Lemma \ref{obv2} we have 
\begin{equation}
\label{transitagain}
\mu(R)\le C(d)\ell(R)^m
\end{equation}
because $R$ (the father of the cube $R_1$) is a transit cube if $R_1$ is terminal.

Now, recalling Lemma \ref{Riesz}, and taking into account that
$\Delta\ci Q f\equiv 0$ outside $Q$, we get
\begin{align*}
\sum_{Q:\,Q\subset
R_1}&
|(\Delta\ci Q f, T\Delta\ci R g)|=
\sum_{Q:\,Q\subset
R_1}
|(\Delta\ci Q f, \chi\ci{R_1}\cdot T\Delta\ci R g)|\\
&\leq C\,\|\chi\ci{R_1}\cdot T\Delta\ci R g \|\ci{L^2(\R^d;\mu)}
\Bigl[\sum_{Q:\,Q\subset
R_1}\|\Delta\ci Q f
\|^2\ci{L^2(\R^d;\mu)}\Bigr]^{\frac{1}{2}}\\
&\leq C\,\|\Delta\ci R g \|\ci{L^2(\R^d;\mu)}
\Bigl[\sum_{Q:\,Q\subset
R_1}\|\Delta\ci Q f
\|^2\ci{L^2(\R^d;\mu)}\Bigr]^{\frac{1}{2}}.
\end{align*}
So, we obtain
$$
\sum_{R:\,R_1\in\mathcal{D}_2^{term}}\,
\sum_{Q:\,Q\subset R_1}
|(\Delta\ci Q f, T\Delta\ci R g)|
$$
$$\leq C\,
\sum_{R:\,R_1\in\mathcal{D}_2^{term}}
\|\Delta\ci R g \|\ci{L^2(\R^d;\mu)}
\Bigl[\sum_{Q:\,Q\subset R_1}
\|\Delta\ci Q f
\|^2\ci{L^2(\R^d;\mu)}\Bigr]^{\frac{1}{2}}
$$
$$
\leq C\,
\Bigl[\sum_{R:\,R_1\in\mathcal{D}_2^{term}}
\|\Delta\ci R g \|^2\ci{L^2(\R^d;\mu)}\Bigr]^{\frac12}
\Bigl[\sum_{R:\,R_1\in\mathcal{D}_2^{term}}\,\,
\sum_{Q:\,Q\subset R_1}
\|\Delta\ci Q f
\|^2\ci{L^2(\R^d;\mu)}\Bigr]^{\frac{1}{2}}.
$$
But the terminal cubes in $\mathcal{D}_2$ do not intersect! Therefore every
$\Delta\ci Q f$ can appear at most once in the last double sum, and we get the bound 
\begin{multline*}
\sum_{R:\,R_1\in\mathcal{D}_2^{term}}\sum_{Q:\,Q\subset R_1}|(\Delta\ci Q f, T^*\Delta\ci R g)|\\\le C\,\Bigl[\sum_{R}\|\Delta\ci R g \|^2\ci{L^2(\R^d;\mu)}\Bigr]^{\frac12}\Bigl[\sum_{Q}\|\Delta\ci Q f\|^2\ci{L^2(\R^d;\mu)}\Bigr]^{\frac{1}{2}}\le C\,\|f\|\ci{L^2(\R^d;\mu)}\|g\|\ci{L^2(\R^d;\mu)}.
\end{multline*}
Lemma \ref{Riesz} has been used again in the last inequality.

\subsection{Estimation of $\sigma_3^{tr}$}
\label{s3transit}
 
 Recall that
 $$
 \sigma_3^{tr}  =\sum_{Q,R\,:\,Q\subset R,\, \ell(Q)<2^{-r}  \ell(R),\\  R\ci Q\text{ is transit}}(\Delta\ci Q f, T^*\Delta\ci R g).
 $$
 Split every term in the sum as
 $$
 (\Delta\ci Q f, T\Delta\ci R g)=(\Delta\ci Q f, T(\chi\ci{R\ci Q}\Delta\ci R g))+(\Delta\ci Q f, T^*(\chi\ci{R\setminus R\ci Q}\Delta\ci R g)).
 $$
 Observe that since $Q$ is good, $Q\subset R$, and $\ell(Q)<2^{-r}  \ell(R)$, we have
 $$
 \dist(Q,\supp \chi\ci{R\setminus R\ci Q}\Delta\ci R g)\ge\dist(Q,sk\, R)\ge \ell(Q)^{\al}\ell(R)^{1-\al}.
 $$
 Using Lemma \ref{FarInteractionLemma} and taking into account that the norm $\|\chi\ci{R\setminus R\ci Q}\Delta\ci R g\|\ci{L^2(\R^d;\mu)}$ does not exceed $\|\Delta\ci R g\|\ci{L^2(\R^d;\mu)}$, we conclude that the sum
 $$
 \sum_{Q,R\,:\,Q\subset R,\, \ell(Q)<2^{-r}  \ell(R),   \\  R\ci Q\text{ is transit}}|(\Delta\ci Q f, T^*(\chi\ci{R\setminus R\ci Q}\Delta\ci R g))|
 $$
 can be estimated by the sum \eqref{twostar}.
 Thus, our task is to find a good bound for the sum
$$
\sum_{Q,R\,:\,Q\subset R,\, \ell(Q)<2^{-r}  \ell(R),
                    \\  R\ci Q\text{ is transit}}
(\Delta\ci Q f, T^*(\chi\ci{R\ci Q}\Delta\ci R g)).
$$

Recalling the definition of $\Delta\ci R g$ and that $R\ci Q$ is a
{\it transit\/} cube, we get
$$
\chi\ci{R\ci Q}\Delta\ci R g=c\ci{R_Q}\chi\ci{R\ci Q},
$$
where
$$
c\ci{R_Q}=\langle g\rangle\ci{R\ci Q}
-
\langle g\rangle\ci{R}
$$
is a {\it constant}.
So, our sum can be rewritten as
$$
\sum_{Q,R\,:\,Q\subset R,\, \ell(Q)<2^{-r}  \ell(R),
                    \\  R\ci Q\text{ is transit}}
c\ci{R_Q}(\Delta\ci Q f, T^*(\chi\ci{R\ci Q})).
$$

Our next goal will be to extend the function $\chi\ci{R\ci Q}$ to the
function $1$ in every term.

Let us observe that
\begin{multline*}
(\Delta\ci Q f, T^*(\chi\ci{\R^d\setminus R\ci Q}))=
\int_{\R^d\setminus R\ci Q}k(x,y)\Delta\ci Q f(x) \,d\mu(x)\,d\mu(y)
\\=
\int_{\R^d\setminus R\ci Q}
[k(x,y)-k(x\ci Q,y)]\Delta\ci Q f(x) \,d\mu(x)\,d\mu(y).
\end{multline*}

Note again that for every $x\in Q$, $y\in \R^d\setminus R\ci Q$, we have
$$
|x\ci Q-y|\ge \frac{\ell(Q)}{2}+\dist(Q,\R^d\setminus R\ci Q)
\geq C\ell(Q)\ge 2|x-x\ci Q|.
$$

Therefore,
$$
|k(x,y)-k(x\ci Q,y)|\le C_{CZ}\frac{|x-x\ci Q|^\tau}{|x\ci Q-y|^{m+\tau}}
\le C_{CZ}\frac{\ell(Q)^\tau}{|x\ci Q-y|^{m+\tau}},
$$
and
$$
|(\Delta\ci Q f, T(\chi\ci{\R^d\setminus R\ci Q}b))|\le
C_{CZ}\ell(Q)^{\tau}\|\Delta\ci Q f\|\ci{L^1(\R^d;\mu)}
\int_{\R^d\setminus R\ci Q}
\frac{d\mu(y)}{|x\ci Q-y|^{m+\tau}}.
$$
Now let us consider the sequence of cubes $R^{(j)}\in\mathcal{D}_2$, beginning with
$R^{(0)}=R\ci Q$ and gradually ascending ($R^{(j)}\subset R^{(j+1)}$,
$\ell(R^{(j+1)})=2\ell(R^{(j)})$) to the starting cube $R^0=R^{(N)}$ of the
lattice $\mathcal{D}_2$. Clearly, all the cubes $R^{(j)}$ are transit cubes.

We have
$$
\int_{\R^d\setminus R\ci Q}
\frac{d\mu(y)}{|x\ci Q-y|^{m+\tau}}=
\int_{R^0\setminus R\ci Q}
\frac{d\mu(y)}{|x\ci Q-y|^{m+\tau}}=
\sum_{j=1}^N
\int_{R^{(j)}\setminus R^{(j-1)}}
\frac{d\mu(y)}{|x\ci Q-y|^{m+\tau}}:=\sum_{j=1}^N I_j.
$$

Note now that, since $Q$ is good and $\ell(Q)<2^{-r}  \ell(R)\le
2^{-r}  \ell(R^{(j)})$ for all $j=1,\dots,N$, we have
$$
\dist(Q,R^{(j)}\setminus R^{(j-1)})\ge
\dist(Q, sk\, R^{(j)})\ge
\ell(Q)^{\al}\ell(R^{(j)})^{1-\al}.
$$
Hence
$$
I_j\le
\frac{1}{[\ell(Q)^{\al}\ell(R^{(j)})^{1-\al}]^{m+\tau}}\int_{R^{(j)}} d\mu.
$$
Recalling that $\al=\frac{\tau}{2(m+\tau)}$, we see that the first factor equals
$$
\dfrac{1}{\ell(Q)^{\frac{\tau}{2}} \ell(R^{(j)})^{m+\frac{\tau}{2}}}\,.
$$
Since $R^{(j)}$ is transit, we have
$$
\int_{R^{(j)}}\,d\mu=\mu(R^{(j)})\le
C\ell(R^{(j)})^m.
$$
Thus,
$$
I_j\le C\frac{1}{\ell(Q)^{\frac{\tau}{2}}\ell(R^{(j)})^{\frac{\tau}{2}}}=C
2^{-(j-1)\frac{\tau}{2}}\frac{1}{\ell(Q)^{\frac{\tau}{2}}\ell(R)^{\frac{\tau}{2}}}.
$$
Summing over $j\ge 1$, we get
$$
\int_{\R^d\setminus R\ci Q}
\frac{d\mu(y)}{|x\ci Q-y|^{m+\tau}}=
\sum_{j=1}^N I_j\le
C\frac{1}{1-2^{-\frac{\tau}{2}}}
\frac{1}{\ell(Q)^{\frac{\tau}{2}}\ell(R)^{\frac{\tau}{2}}}.
$$

Now let us note that
\begin{equation}
\label{cQR}
|c_{R_Q}| \le \frac{\|\Delta_R g\|_{L^1(R_Q,\mu)}}{\mu(R_Q)}\le \frac{\|\Delta_R g\|_{L^2(R_Q,\mu)}}{\sqrt{\mu(R_Q)}}\,.
\end{equation}

We finally obtain
\begin{eqnarray*}
|(\Delta\ci Q f, T^*(\chi\ci{\R^d\setminus R\ci Q}))| & \leq & \frac{C}{1-2^{-\frac{\tau}{2}}}
\left[\frac{\ell(Q)}{\ell(R)}\right]^{\frac{\tau}{2}}
\sqrt{\frac{\mu(Q)}{\mu(R\ci Q)}}
\|\Delta\ci Q f\|\ci{L^2(\R^d;\mu)}
\|\Delta\ci R g\|\ci{L^2(\R^d;\mu)}
\end{eqnarray*}
and
\begin{multline*}
\sum_{Q,R\,:\,Q\subset R,\, \ell(Q)<2^{-r}  \ell(R), R\ci Q\text{is transit}}
|c\ci{R,Q}|\cdot|(\Delta\ci Q f, T^*(\chi\ci{\R^d\setminus
R\ci Q}))|
\\
\leq
C\frac{1}{1-2^{-\frac{\tau}{2}}}
\sum_{j=1}^{2^d}\,\sum_{Q,R\,:\,Q\subset R_j}
\left[\frac{\ell(Q)}{\ell(R)}\right]^{\frac{\tau}{2}}
\sqrt{\frac{\mu(Q)}{\mu(R_j)}}
\|\Delta\ci Q f\|\ci{L^2(\R^d;\mu)}
\|\Delta\ci R g\|\ci{L^2(\R^d;\mu)}.
\end{multline*}

We next define a linear operator on $\ell^2\to \ell^2$ given by
$$
T_{Q,R}:=\left[\frac{\ell(Q)}{\ell(R)}\right]^{\frac{\tau}{2}}\sqrt{\frac{\mu(Q)}{\mu(R_1)}}\quad (Q\subset R_1).
$$
Key to the rest of this section is demonstrating that this is a bounded operator on $\ell^2$.

\begin{lemma}
\label{TQRshortrange}
For any two
families  $\{a\ci Q\}\ci{Q\in\mathcal{D}_1^{tr}  }$ and $\{b\ci R\}\ci{R\in\mathcal{D}_2^{tr}  }$
of nonnegative numbers, one has
$$
\sum_{Q,R:Q\subset R_1}T\ci{Q,R}a\ci Q b\ci R\le
\frac{1}{ 1-2^{-\frac{\tau}{2}} }
\Bigl[\sum_{Q}a\ci Q^2\Bigr]^{\frac12}
\Bigl[\sum_{R}b\ci R^2\Bigr]^{\frac12}.
$$
\end{lemma}

\begin{proof}
Let us ``slice'' the matrix $T\ci{Q,R}$ according to the ratio
$\frac{\ell(Q)}{\ell(R)}$. Namely, let
$$
T^{(k)}_{Q,R}=\left\{
\aligned
T\ci{Q,R},&\qquad\text{if } Q\subset R_1,\  \ell(Q)=2^{-k}\ell(R);
\\
0,&\qquad\text{otherwise}
\endaligned
\right.
$$
($k=1,2,\dots$).
It is enough to show that for every $k\ge 0$,
$$
\sum_{Q,R}T^{(k)}\ci{Q,R}a\ci Q b\ci R\le
2^{-\frac{\tau}{2} k }
\Bigl[\sum_{Q}a\ci Q^2\Bigr]^{\frac12}
\Bigl[\sum_{R}b\ci R^2\Bigr]^{\frac12}.
$$
The matrix $\{T^{(k)}_{Q,R}\}$ has a very good ``block'' structure:
every $a\ci Q$ can interact with {\it only one}
 $b\ci R$.
So, it is enough to estimate each block separately, i.e., to show that
for every fixed $R\in\mathcal{D}_2^{tr}  $,
$$
\sum_{Q:\,Q\subset R_1,\, \ell(Q)=2^{-k}\ell(R)}
2^{-\frac{\tau}{2} k }
\sqrt{\frac{\mu(Q)}{\mu(R_1)}}
a\ci Q b\ci R\le
2^{-\frac{\tau}{2} k }
\Bigl[\sum_{Q}a\ci Q^2\Bigr]^{\frac12}
b\ci R.
$$
Because the cubes $Q\in \mathcal{D}_1$ of fixed size do not intersect, we have that 
$$
\sum_{Q:\,Q\subset R_1,\, \ell(Q)=2^{-k}\ell(R)}
\mu(Q)\le \mu(R_1).
$$

Reducing both parts by the non-essential factor $2^{-\frac{\tau}{2} k }b\ci R$ and using the above observations we see that the desired estimate follows from 

\begin{eqnarray*}
\sum_{Q:\,Q\subset R_1,\, \ell(Q)=2^{-k}\ell(R)}\sqrt{\frac{\mu(Q)}{\mu(R_1)}} a\ci Q & \leq & \Bigl[
\sum_{Q:\,Q\subset R_1,\, \ell(Q)=2^{-k}\ell(R)}
\frac{\mu(Q)}{\mu(R_1)}\Bigr]^{\frac12}
\Bigl[\sum_{Q}a\ci Q^2\Bigr]^{\frac12}\\
& \leq & \Bigl[\sum_{Q}a\ci Q^2\Bigr]^{\frac12}.
\end{eqnarray*}
This completes the proof.

\end{proof} 

\medskip

\begin{remark}
We did not use here the fact that the sequences $\{a_Q\}$ and $\{b_R\}$ are supported on transit cubes. We actually proved the following Lemma.

\begin{lemma}
\label{actually}
The matrix $\{T\ci{Q,R}\}$ defined
by
$$
T\ci{Q,R}:=\left[\frac{\ell(Q)}{\ell(R)}\right]^{\frac{\tau}{2}}
\sqrt{\frac{\mu(Q)}{\mu(R_1)}}\qquad\quad(Q\subset R_1),
$$
generates a bounded operator in $l^2$.
\end{lemma}
\end{remark}

\medskip

We just finished estimating an extra term which appeared when we extended $\chi\ci{R\ci Q}$ to $1$. So, the extension of $\chi\ci{R\ci Q}$ to the whole $1$ does not cause too
much harm, and we get the sum
$$
\sum_{Q,R\,:\,Q\subset R,\, \ell(Q)<2^{-r}  \ell(R),
                    \\  R\ci Q\text{ is transit}}
c\ci{R_Q}(\Delta\ci Q f, T^* 1)
$$
to estimate. Note that the inner product
$(\Delta\ci Q f, T^* 1)$ {\it does not depend\/} on $R$ at all, so it seems to be a good idea to sum over $R$ for fixed $Q$ first.

Recalling that

$$
c\ci{R_Q}=\langle  g\rangle\ci{R\ci Q}-\langle g\rangle\ci{R}
$$
and that $\Lambda g=0 \Longleftrightarrow \langle g\rangle\ci{R^0}=0$, we
conclude that for every $Q\in\mathcal{D}_1^{tr}  $ that really appears in the above sum,
$$
\sum_{R\,:\,R\supset Q,\, \ell(R)>2^{m}\ell(Q),                    \\  R\ci Q\text{ is transit}}c\ci{R_Q}=\langle g\rangle\ci{R_Q}\,.
$$

\vspace{.1in}

\begin{defi}
Let $R(Q)$ be the smallest {\it transit} cube $R\in \mathcal{D}_2$ containing $Q$
and such that $\ell(R)\ge 2^r\ell(Q)$.
\end{defi}

\vspace{.1in}

So, we obtain the sum

$$
\sum_{Q:\,\ell(Q)<2^{-r}  \ell(R)} \langle g\rangle\ci{R(Q)}(\Delta\ci Q f, T^* 1)
$$
to take care of.

\medskip

\begin{remark}
Let us recall that we had the convention that  says that $Q$ here are only good ones and, of course, they are only transit cubes. The range of summation should be $Q\in \mathcal{D}_1^{tr}  $, $Q$ is good (default); there exists a cube $R\in\mathcal{D}_2^{tr}  $ such that $\ell(Q)<2^{-r}  \ell(R)$, $Q\subset R$ and the child $R\ci Q$ (the one containing $Q$) of $R$  is transit. In other words, in fact, the sum is written formally incorrectly. We have to replace $R(Q)$ by $R_Q$ in the summation. However, the smallest transit cube containing $Q$ (this is $R(Q)$) and the smallest transit child (containing $Q$) of a certain subcube $R$ of $R^0$  (this child is $R_Q$) are of course the one and the same cube, unless $R(Q)=R^0$. Thus the sum formally has some extra terms corresponding to $R(Q) =R^0$,  but fortunately they all are zeros!  Since we work now with $g$ such that $\Lambda g=0$ (recall that $\Lambda g$ denotes the average of $g$ with respect to $\mu$), so $\langle g\rangle_{R(Q)} =0$ if $R(Q)=R^0$.
\end{remark}

\medskip

\subsection{Pseudo-$BMO$ and a Special Paraproduct}
\label{bmo}

To introduce the paraproduct operator, we rewrite our sum as follows
\begin{align*}
\sum_{Q:\,\ell(Q)<2^{-r}  \ell(R)}\langle g\rangle\ci{R(Q)}(\Delta\ci Q f, T^*1)&= \sum_{Q:\,\ell(Q)<2^{-r}  \ell(R)}\langle g\rangle\ci{R(Q)}( f,\Delta_Q^* T^*1) \\&= ( f, \sum_{Q:\,\ell(Q)<2^{-r}  \ell(R)}\langle g\rangle\ci{R(Q)}\Delta_Q T^*1)\,.
\end{align*}
Here, we used the fact that $\Delta_Q^* =\Delta_Q$.  The term that appears in the pairing with $f$ is ubiquitous enough that it deserves to be singled out.

\begin{defi}
Given a function $F$, we define the paraproduct operator with symbol $F$ that acts on $L^2(\R^d;\mu)$ functions via 
$$
\Pi_{F}g(x) := \sum_{R\in \mathcal{D}_2,\,R\subset R^0}\langle g\rangle_R\sum_{Q\in\mathcal{D}_1^{tr}\cap\mathcal{G}_1,\,\ell(Q)= 2^{-r}  \ell(R)} \Delta_Q F(x)\,.
$$
\end{defi}

As in the case of homogeneous harmonic analysis, the behavior of the paraproduct operator $\Pi_F$ is dictated by the membership of $F$ in a suitable ``$BMO$'' space.

\begin{defi}
A function $F\in L^2(\R^d;\mu)$ belongs to ``pseudo-$BMO$'' if there exists $\lambda>1$ such that for any cube $Q$ such that $\mu(sQ)\le K\, s^m\ell(Q)^m, \, s\ge 1,$ we have
$$
\int_Q |F(x)-\langle F\rangle_Q|^2\,d\mu(x) \le C\, \mu(\lambda\,Q)\,.
$$
\end{defi}

We next demonstrate that the image of $1$ under the action of $T^*$ is an element of pseudo-$BMO$.

 \begin{lemma}
\label{Tstar1}
Let $\mu$, $T$ satisfy the assumptins of Theorem \ref{md}. Then
\begin{equation}
\label{para}
T^*1 \in \textnormal{pseudo-BMO}\,.
\end{equation}
Here $C$ depends only on the constants of  Theorem \ref{md}.
\end{lemma}

\begin{proof}
For $x\in Q$ we write $T^*1(x) = (T^*\chi_{\lambda Q})(x) + (T^*\chi_{\R^d\setminus\lambda Q})(x)=: \vf(x) +\psi(x)$.
First, we notice that 
$$
x,y \in Q\Rightarrow |\psi(x)-\psi(y)| \le C(K,\lambda,\tau)\,,
$$ where $K$ is the constant from our definition above. This is easy:
\begin{eqnarray*}
|\psi(x)-\psi(y)| & \le &  \int_{\mathbb{R}^d\setminus \lambda Q} |k(x,t)-k(y,t)| \,d\mu(t)= \sum_{j=1}^{\infty} \int_{\lambda^{j+1} Q\setminus \lambda^j Q} |k(x,t)-k(y,t)| \,d\mu(t)\\
 & \le &  \sum_{j=1}^{\infty} \frac{\ell(Q)^{\tau}}{(\lambda^j\ell(Q))^{m+\tau}}\, K (\lambda^j\ell(Q))^{m}\le C(K,\lambda, \tau)\,.
\end{eqnarray*}
Therefore,
$$
\int_Q |\psi(x)-\langle \psi\rangle_Q|^2\,d\mu(x) \le C\, \mu(Q)\,.
$$
But, 
$$
\int_Q |\vf(x)-\langle \vf\rangle_Q|^2\,d\mu(x) \le 4\int_Q |T^*\chi_{\lambda Q}(x)|^2\,d\mu(x) \lesssim A\,\mu(\lambda Q)
$$
by the $T1$ assumption, Condition (\ref{T1}), of Theorem \ref{md}.

\end{proof}

\begin{lemma}
\label{paraproductlm}
Let $\mu$, $T$ satisfy the assumptions of Theorem \ref{md}. Then
\begin{equation}
\label{paraproduct}
\|\Pi_{T^*1}\|_{L^2(\R^d;\mu)\to L^2(\R^d;\mu)} \leq C\,.
\end{equation}
Here $C$ depends only on the constants of  Theorem \ref{md}.
\end{lemma}

\begin{proof}

Let $F=T^*1$. In the definition of $\Pi_F$ since all $\Delta_Q$ are mutually orthogonal, it is easy to see that
$$
\|\Pi_F g\|_{L^2(\R^d;\mu)}^2 = \sum_{R\in \mathcal{D}_2,\,R\subset R^0}|\langle g\rangle_R|^2\sum_{Q\in\mathcal{D}_1^{tr}\cap\mathcal{G}_1,\,\ell(Q)= 2^{-r}  \ell(R)} \|\Delta_Q F\|_{L^2(\R^d;\mu)}\,.
$$
Put 
$$
a_R:= \sum_{Q\in\mathcal{D}_1^{tr}\cap\mathcal{G}_1,\,\ell(Q)= 2^{-r}  \ell(R)} \|\Delta_Q F\|_{L^2(\R^d;\mu)}\,.
$$
By the Carleson Embedding Theorem, it is enough to prove that for every $S\in \Dk_2$

\begin{equation}
\label{cet}
\sum_{R\in \Dk_2, R\subset S} a_R \le C\,\mu(S)\,,
\end{equation}
or equivalently, 

\begin{equation}
\label{cet1}
\sum_{Q\in\mathcal{D}_1^{tr},\,\ell(Q)\le 2^{-r}  \ell(R), \dist(Q, \pd R) \ge\ell(Q)^{\alpha}\ell(R)^{1-\alpha}} \|\Delta_Q F\|_{L^2(\R^d;\mu)} \le C\, \mu(R)\,.
\end{equation}

Perform a Whitney decomposition of $R$ into disjoint cubes $P$, such that $1.5 P \subset R$, and $1.4 P$ have only bounded multiplicity $C(d)$ of intersection.  Then, consider the sums
\begin{equation}
\label{cet2}
s_P:=\sum_{Q\in\mathcal{D}_1^{tr},\,\ell(Q)\le 2^{-r}  \ell(R), Q\cup P\neq \emptyset, \dist(Q, \pd R) \ge\ell(Q)^{\alpha}\ell(R)^{1-\alpha}} \|\Delta_Q F\|_{L^2(\R^d;\mu)}\,.
\end{equation}

This $s_P$ can be zero if there is no transit cubes as above intersecting it. But if $s_P\neq 0$ then necessarily
$$
\mu(P)\le C(d)\ell(P)^m\,,
$$ 
and moreover
$$
\mu(sP)\le C(d)\,s^m\,\ell(P)^m\,,\,\,\forall s\ge 1\,.
$$

In fact, in this case $P$ intersects a transit cube $Q$, which by elementary geometry is ``smaller" than $P$, so $\ell(Q) \le C(r,d) \ell(P)$. But, then the above inequalities follow from the definition of {\it transit}.

It is also clear that for large $r$ and for $Q$ and $P$ as above
$$ 
Q\cap P\neq \emptyset\Rightarrow Q\subset 1.2 \,P\,.
$$

Therefore,
$$
s_P\neq 0\Rightarrow s_P \le \sum_{Q\in\mathcal{D}_1^{tr},\,\ell(Q)\le 2^{-r}  \ell(R), Q\subset 1.2\, P\, \dist(Q, \pd R) \ge\ell(Q)^{\alpha}\ell(R)^{1-\alpha}} \|\Delta_Q F\|_{L^2(\R^d;\mu)}\,.
$$
So,

$$
s_P\neq 0\Rightarrow s_P \le \int_{1.2P} |F(x)-\langle F\rangle_{1.2\,P}|^2\,d\mu(x)\le C\,\mu(1.4\, P)\,.
$$
The last inequality follows from Lemma
\ref{Tstar1}.

Now we add all $s_P$'s and we get that this is less than or equal to $C\, \sum \mu (1.4\,P)$. This is smaller than $C_1\, \mu(R)$ as $1.4P$'s have multiplicity of intersection $C(d)<\infty$.

\end{proof}

\section{The Diagonal Sum: Controlling Term $\sigma_{1}$.}
\label{diag}

To complete the estimate of $|(Tf_{good},g_{good})|$ in only remains to estimate $\sigma_1$. But notice that
$$
\|\Delta_Q f\|_{L^1(\R^d;\mu)} \le \|\Delta_Q f\|_{L^2(\R^d;\mu)} \sqrt{ \mu(Q)},\quad \|\Delta_R g\|_{L^1(\R^d;\mu)} \le \|\Delta_R g\|_{L^2(\R^d;\mu)} \sqrt{ \mu(R)}\,.
$$ 

Remember that all cubes $Q$ and $R$ in all sums are transit cubes. In particular, in $\sigma_1$ we have that $Q$ and $R$ are close and of almost the same size. If a son of $Q$, $S(Q)$, is terminal, then by Lemma \ref{obv1}
$$
|(T\chi_{S(Q)}\Delta_Q f, \Delta_R g)|\le \frac{\sqrt{ \mu(Q)}\sqrt{ \mu(R)}}{\ell(Q)^m} \|\Delta_Q f\|_{L^2(\R^d;\mu)}  \|\Delta_R g\|_{L^2(\R^d;\mu)} \,.
$$
So if a son of $Q$ is terminal but $Q$ and $R$ are transit, then $\mu(Q) \le C\, \ell(Q)^m\approx\ell(R)^m$.
Summing such pairs (and symmetric ones, where a son of $R$ is terminal) we get $C(r)\|f\|_{L^2(\R^d;\mu)}|\|g\|_{L^2(\R^d;\mu)}$.

We are left with the part of $\sigma_1$, where we sum over $Q$ and $R$ such that their sons are transit. Then we use the pairing
$$
|(T\chi_{S(Q)}\Delta_Q f, \chi_S(R)\Delta_r g)|\le A^2|c_{S(Q)}||c_{S(R)}|\sqrt{\mu(S(Q))}\sqrt{\mu(S(R))}\,.
$$ 
The estimate above follows from our $T1$ assumption in Theorem \ref{md}.
Now use \eqref{cQR} to obtain
$$
|(T\chi_{S(Q)}\Delta_Q f, \chi_{S(R)}\Delta_R g)|\le C\, \|\Delta_Q f\|_{L^2(\R^d;\mu)}  \|\Delta_R g\|_{L^2(\R^d;\mu)} \,.
$$

This completes the estimate of term $\sigma_1$ and thus the proof of Theorem \ref{md}.  

\bigskip

\section{The Proof of Theorem \ref{md1}}

The proof of Theorem \ref{md1} repeats verbatim that of Theorem \ref{md}. We need only to provide ourselves with the family of ``dyadic" lattices in metric $\Delta$ of cubes $Q$, which satisfies the following property. 

First, consider the case when $d=2$.  Then, we have that the metric $\Delta(x,y)$ is comparable to  the following  quasi-metric.   Let $z\in\{z\in \mathbb{C}^{d}: 1/2\le |z|\le 2\}$ be thought of as $(|z|, \frac{z}{|z|})$, and then we have $\Delta\left(\left(|z|, \frac{z}{|z|}\right), \left(|w|,\frac{w}{|w|}\right)\right) = ||z|-|w||+\rho(\frac{z}{|z|}, \frac{w}{|w|})$, where $\rho(\xi,\eta)=|1-\bar{\xi}\cdot\eta|$ is the spherical metric between points on the boundary of the unit ball.  For a point on the sphere we can introduce local coordinates $p=(r, c)$, where r is real and c is a complex number $\rho(p_1, p_2) = |r_1-r_2| + |c_1-c_2|^2$.  Thus, the metric $\Delta(x,y)$ is equivalent to the quasi-metric in $\mathbb{R}^4$, given by $\lambda(x,y)=|x_1-y_1|+|x_2-y_2|+|x_3-y_3|^2+|x_4-y_4|^2$.  

As can be seen from the proof of Theorem \ref{md}, the main result will still hold for Calder\'on--Zygmund kernels with respect to this quasi-metric.  The only issue that is not clear is the construction of the dyadic lattices.  To obtain these, one starts with the unit cube and sub-divides two of the sides by a factor of 2, while the other two sides are subdivided by a factor of 4 (these correspond to the structure of the metric $\lambda$).  This can then be repeated.  These are then shifted using the standard translations in $\mathbb{R}^4$.  This then gives the dyadic lattices of interest.  We then are only left with checking the probability statements (Theorem \ref{badprob}), and this is a straightforward modification of what appears in this paper.  We sketch the details now.

In this case, the geometry of a cube $Q$ is governed by a parameter $\mathfrak{q}$, which determines two of its sides (the other two are then $\mathfrak{q}^2$).  Similarly for $R$, it has a parameter $\mathfrak{r}$ determining its geometry. 

\medskip 

\begin{remark}
Notice that $\lambda$-diameters  ($\lambda$-size $\ell(Q),\ell(R)$, so to speak) of $Q, R$ are equivalent  correspondingly to $\mathfrak{q}^2, \mathfrak{r}^2$.  This is because of the structure of the quasi-metric $\lambda$.
\end{remark}

\medskip

Again, one would consider the parameters of the \cz\, kernel $\tau$ and $m$ (but with respect to the metric $\Delta$), and set $\alpha=\frac{\tau}{2\tau+2m}$.  We fix a small number $\delta>0$, and $S\geq 2$ chosen momentarily, and choose an integer $r$ such that
$$
2^{-r}\leq \delta^S<2^{-r+1}.
$$

Then a cube $Q$ is  called $\delta$-bad if there exists a cube $R$ such that
\begin{itemize}
\item[(i)] $\mathfrak{r}\geq 2^r\mathfrak{q}$;
\item[(ii)] $\textnormal{dist}_{\lambda}(Q,\partial R)<\mathfrak{q}^{2\alpha}\mathfrak{r}^{2(1-\alpha)}$.
\end{itemize}

One then can prove the following Proposition

\begin{prop}
One can choose $S= S(\alpha)$ in such a way that for any fixed $Q\in \mathcal{D}_1$,
\begin{equation}
\mathbb{P}_{\om_2}\{Q \,\text{is bad}\} \leq \delta^2\,.
\end{equation}
By symmetry $\mathbb{P}_{\om_1}\{R \,\text{is bad}\} \leq \delta^2$ for any fixed $R\in \mathcal{D}_2$.
\end{prop}

The key computation to observe to prove this Proposition is the following.  Fix an integer $k\geq r$, and let us estimate the probability that there exists a cube $R\in\mathcal{D}_2$ of parameter $\mathfrak{r}=2^k\mathfrak{q}$  such that $\dist_{\lambda}(Q, \partial R)\le \ell(Q)^{\al}\ell(R)^{1-\al}=\mathfrak{q}^{2\al}\mathfrak{r}^{2(1-\al)}$.  Geometry shows that it is equal to the ratio of the (usual) volume of the narrow strip of $\lambda$-distance $2^{2(1-\alpha)k}\mathfrak{q}^2$ around the boundary of the cube $R$ to the whole volume of the cube with parameter $2^k \mathfrak{q}$.  

The volume of the strip around the boundary is the sum of slabs adjoint to the faces of cube $R$ of parameter $ \mathfrak{r}=2^k \mathfrak{q}$. There are faces of two types: type (1)--it has (Euclidean!) measurements $\mathfrak{r}\times \mathfrak{r}\times \mathfrak{r}^2$, type (2)--it has (Euclidean!) measurements $\mathfrak{r}\times \mathfrak{r}^2\times \mathfrak{r}^2$.  The slabs consist of those points, whose $\lambda$-distance to faces is at most 

$$
\epsilon:=\mathfrak{q}^{2\al}\mathfrak{r}^{2(1-\al)}\,.
$$
 For faces of type (1), the slab with this property has (Euclidean!) thickness at most $C\,\epsilon$. For faces of type (2), the slab will obviously have the (Euclidean!) thickness at most $C\,\epsilon^{1/2}$.  Consequently, this implies the volume of the strip is at most 
$$
C\,(\epsilon \cdot \mathfrak{r}^{2}\cdot \mathfrak{r}\cdot\mathfrak{r}+ \epsilon^{1/2}\cdot \mathfrak{r}\cdot \mathfrak{r}^{2}\cdot \mathfrak{r}^{2})\,.
$$
Plugging  $\mathfrak{r}=2^k\mathfrak{q}$ and $\epsilon= 2^{k(2-2\al)}\mathfrak{q}$ into the above we obtain that the volume of the strip is at most 
$$
C\,(2^{k(2-2\al +4)} + 2^{1-\al +5})\mathfrak{q}^6\,.
$$
Or, it is at most $C\mathfrak{q}^{6} 2^{6k-\alpha k}$ (here $C$ is an absolute constant), while the volume  of the cube with parameter $2^k\mathfrak{q}$ is $2^{6k}\mathfrak{q}^6$.  We conclude that this ratio is less than
$C\,2^{-k\al}$.  Therefore, the probability that the cube $Q$ is bad  does not exceed
$$
C\,\sum_{k=r}  ^\infty 2^{-k\al}=C\frac{2^{-r\al}}{1-2^{-\al}}\,.
$$
Since the integer $r$ was chosen so that $2^{-r}\leq \delta^S<2^{-r+1}$ we will then simply choose the minimal $S=S(\alpha)$ such that $\frac{A\, \delta^{S\al}}{1-2^{-\al}}\leq \delta^2$ (of course, $S=3/\al$ is enough for all small $\delta$'s).  This construction of course works when $d=2n$ with appropriate modifications  to the geometry of the sets. 

\section{Concluding Remarks}
The considerations explored in this paper can be extended to the case of metric spaces.  This was carried out by the authors in \cite{VW2}.  There is also recent related work by Hyt\"onen and Martikainen, \cites{HM1, HM2} where they prove Theorems of a similar nature.  In particular the show ``Tb'' theorems on quasimetric spaces equipped with upper doubling measures.  As an application of their techniques, they are able to obtain the same characterization of Carleson measures for the $B_\sigma^2(\mathbb{B}_{2d})$, though, not our abstract Theorem \ref{md}.

Finally, we remark that it is anticipated that the method developed in this paper, and in addition the tools of non-homogeneous harmonic analysis, should have other immediate applications.

\end{document}